\documentclass[12pt]{amsart}

\usepackage{multirow}

\usepackage[english]{babel}
\usepackage[utf8x]{inputenc}
\usepackage[T1]{fontenc}
\usepackage{mathrsfs}
\usepackage{amsmath,amssymb,dsfont,mathrsfs}
\usepackage{amsthm}

\usepackage{pgf,tikz,pgfplots}
\pgfplotsset{compat=1.13}
\usetikzlibrary{arrows}
\usepackage{multicol}
\usepackage{bm}
\usepackage[all,cmtip]{xy}
\usepackage{tcolorbox}
\usepackage{tikz-cd}

\usepackage{mathtools}
\usepackage{colonequals}
\usepackage[letterpaper,top=1in,bottom=1in,left=1in,right=1in]{geometry}

\usepackage[autostyle]{csquotes}

\usepackage{enumitem}

\newenvironment{enumalphii}
{\begin{enumerate}}
{\end{enumerate}}

\newenvironment{enumroman}
{\begin{enumerate}}
{\end{enumerate}}

\newenvironment{enumalg}
{\begin{enumerate}}
{\end{enumerate}}

\usepackage{pgfplots}
\pgfplotsset{compat=newest}

\newcommand{\mathbbords}{\mathbb}

\newcommand{\C}{{\mathbbords{C}}}
\newcommand{\N}{{\mathbbords{N}}}
\newcommand{\Q}{{\mathbbords{Q}}}

\newcommand{\R}{{\mathbbords{R}}}

\newcommand{\Z}{{\mathbbords{Z}}}

\newcommand{\F}{\mathbbords{F}}
\newcommand{\PP}{\mathbbords{P}}
\newcommand{\frakN}{\mathfrak{N}}

\newcommand{\frakp}{\mathfrak{p}}

\DeclareMathOperator{\ord}{ord}
\DeclareMathOperator{\Tr}{Tr}

\DeclareMathOperator{\SL}{SL}
\DeclareMathOperator{\Nm}{Nm}
\DeclareMathOperator{\GL}{GL}
\DeclareMathOperator{\PSL}{PSL}
\DeclareMathOperator{\PGL}{PGL}
\DeclareMathOperator{\PXL}{PXL}
\DeclareMathOperator{\Sym}{Sym}

\DeclareMathOperator{\adim}{adim}
\DeclareMathOperator{\M}{M}

\DeclareMathOperator{\lcm}{lcm}

\newcommand{\calH}{\mathcal{H}}

\renewcommand{\triangle}{\Delta}

\renewcommand{\P}{\ensuremath{\mathbbords{P}}}

\usepackage[pagebackref]{hyperref}
\hypersetup{colorlinks=true,urlcolor=blue,citecolor=blue,linkcolor=blue}

\renewcommand*{\backref}[1]{}
\renewcommand*{\backrefalt}[4]{%
  \ifcase #1 %
    \relax
  \or
    $\uparrow$#2.%
  \else
    $\uparrow$#2.%
  \fi%
}

\usepackage{fancyhdr}
\usepackage{color}
\definecolor{qqzzcc}{rgb}{0.,0.6,0.8}

\usepackage[nameinlink]{cleveref}

\numberwithin{equation}{section}

\newtheorem{theorem}[equation]{Theorem}
\newtheorem{proposition}[equation]{Proposition}

\newtheorem{conj}[equation]{Conjecture}
\newtheorem{lemma}[equation]{Lemma}
\newtheorem{corollary}[equation]{Corollary}

\theoremstyle{definition}
\newtheorem{definition}[equation]{Definition}
\newtheorem{example}[equation]{Example}
\newtheorem{algorithm}[equation]{Algorithm}

\theoremstyle{remark}
\newtheorem{remark}[equation]{Remark}

\newcommand{\defi}[1]{\textsf{#1}}

\newcommand{\Pone}[2]{{\renewcommand*{\arraystretch}{0.45} \begin{pmatrix} {#1} \\ \rotatebox[origin=c]{90}{:} \\ {#2} \end{pmatrix} \renewcommand*{\arraystretch}{1}}}

\def\Legendre(#1,#2){%
\begin{pmatrix}
#1\cr 
\hline
#2\cr
\end{pmatrix}}

\newcommand{\frakd}{\mathfrak d}
\newcommand{\calO}{\mathcal O}

\newcommand{\psmod}[1]{~(\textup{\text{mod}}~{#1})}

\DeclareMathOperator{\discrd}{discrd}

\DeclareMathOperator{\nrd}{nrd}

\title{Triangular modular curves of small genus}

\author{Juanita Duque-Rosero}
\address{Department of Mathematics, Dartmouth College, 6188 Kemeny Hall, Hanover, NH 03755, USA}
\email{juanita.duque.rosero.gr@dartmouth.edu}

\author{John Voight}
\address{Department of Mathematics, Dartmouth College, 6188 Kemeny Hall, Hanover, NH 03755, USA}
\email{jvoight@gmail.com}

\date{}

\begin{document}
\maketitle
\begin{abstract}
    Triangular modular curves are a generalization of modular curves that arise from quotients of the upper half-plane by congruence subgroups of hyperbolic triangle groups. These curves also arise naturally as a source of Belyi maps with monodromy $\text{PGL}_2(\mathbbords{F}_q)$ or $\text{PSL}_2(\mathbbords{F}_q)$. We present a computational approach to enumerate Borel-type triangular modular curves of low genus, and we carry out this enumeration for prime level and small genus.
\end{abstract}

\section{Introduction}\label{sec:intro}

\subsection*{Motivation}

The study of modular curves has rewarded mathematicians for perhaps a century.  For an integer $N \geq 1$, let $\Gamma_0(N),\Gamma_1(N) \leq \SL_2(\Z)$ be the usual congruence subgroups and let $X_0(N),X_1(N)$ be the corresponding quotients of the completed upper half-plane.  The genera of $X_0(N)$ and $X_1(N)$ as compact Riemann surfaces can be computed using the Riemann--Hurwitz formula, and it can readily be seen that there are only finitely many of any given genus $g \geq 0$. 

The study of modular curves of small genus goes back at least to Fricke \cite[p.~357]{Fricke}.  At the end of the twentieth century, Ogg enumerated and studied elliptic \cite{Ogg1} and hyperelliptic \cite{Ogg2} modular curves; the resulting Diophantine study \cite{Ogg3} informed Mazur's classification of rational isogenies of elliptic curves \cite{Mazur}, where the curves of genus $0$ are precisely the ones with infinitely many rational points.  This explicit study continues today, extended to include all quotients of the upper half-plane by congruence subgroups of $\SL_2(\Z)$; the list up to genus $24$ was computed by Cummins--Pauli \cite{CumminsPauli}.  Recent papers have studied curves with infinitely many rational points in the context of Mazur's \emph{Program B}---see Rouse--Sutherland--Zureick-Brown \cite{RSZB} for further references and recent results in this direction.

Given this rich backdrop, it is worthwhile to pursue  generalizations.  For example, replacing $\SL_2(\Z)$ with its quaternionic cousins, Voight \cite{Voight} enumerated all Shimura curves of the form $X_0^1(\mathfrak{D},\mathfrak{M})$ of genus at most $2$.  In a similar direction, Long--Maclachlan--Reid \cite{LMR} enumerated all maximal arithmetic Fuchsian groups of genus $0$ over $\Q$, corresponding to quotients of Shimura curves by the full group of Atkin--Lehner involutions.

\subsection*{Setup and main result}

In this paper, we consider a different type of generalization: namely, from the point of congruence subgroups of triangle groups as introduced by Clark--Voight \cite{ClarkVoight}.  We briefly introduce this construction; for more detail, see \cref{sec:curveDef}. 

Let $a,b,c \in \Z_{\geq 2} \cup \{\infty\}$, and suppose that $1/a+1/b+1/c < 1$ (where $1/\infty=0$).  Then there is a triangle in the upper half-plane $\mathcal{H}$ (completed if $\infty \in \{a,b,c\}$) with angles $\pi/a$, $\pi/b$, and $\pi/c$, unique up to isometry.  The reflections in the sides of this triangle generate a discrete subgroup of $\PGL_2(\R)$, and the orientation-preserving subgroup (of index $2$) defines the \defi{triangle group} $\Delta=\Delta(a,b,c) \leq \PSL_2(\R)$, with presentation
\[ \Delta(a,b,c) = \langle \delta_a,\delta_b,\delta_c \,|\, \delta_a^a = \delta_b^b = \delta_c^c = \delta_a\delta_b\delta_c=1 \rangle \]
(omitting the relation $\delta_s^s$ when $s=\infty$).  The triangle group acts properly by isometries on $\calH$ and the quotient $X(a,b,c) \colonequals \Delta(a,b,c) \backslash \calH$ can be given the structure of a compact Riemann surface of genus $0$, isomorphic to $\PP^1$ with a unique coordinate $t$ taking values $0,1,\infty$ at the vertices labelled $a,b,c$, respectively.  For example, we recover the classical modular group as $\Delta(2,3,\infty) \simeq \PSL_2(\Z)$, with coordinate $t=j/1728$.

Let $m \colonequals \gcd(\{a,b,c\} \smallsetminus \{\infty\})$, with $m=1$ for $a=b=c=\infty$.  Attached to $(a,b,c)$ is an extension \begin{equation}
E=E(a,b,c) \subseteq F=F(a,b,c) \subseteq \Q(\zeta_{2m})^+
\end{equation}
of totally real, abelian number fields.  The field $F$ is the subfield of $\R$ generated by $\Tr \Delta$, and similarly $E$, called the \defi{invariant trace field}, is the subfield generated by $\Tr \Delta^{(2)}$ where $\Delta^{(2)} \leq \Delta$ is the subgroup generated by squares.  Let $\Z_E \subset E$ be the ring of integers and similarly $\Z_F \subset F$.

Let $\frakN \subseteq \Z_E$ be a nonzero ideal.  Then there is a natural reduction homomorphism $\varpi_{\frakN}$ with domain $\Delta$, intuitively thought of as reducing matrix entries modulo $\frakN$ but with a rigorous quaternionic interpretation (see below).  The kernel $\Gamma(a,b,c;\frakN) \colonequals \ker \varpi_{\frakN}$ is called the \defi{principal congruence subgroup of level $\frakN$}.  A subgroup $\Gamma \leq \Delta(a,b,c)$ is said to be \defi{congruence} if $\Gamma \geq \Gamma(a,b,c;\frakN)$ for some $\frakN$; the \defi{level} of a congruence subgroup is the minimal such $\frakN$.  Given a congruence subgroup $\Gamma \leq \Delta(a,b,c)$, we call the quotient $X(a,b,c;\Gamma) \colonequals \Gamma \backslash \calH$ a \defi{triangular modular curve}, since they generalize the classical modular curves.  The quotient map
\begin{equation}
\varphi_\frakN \colon X(a,b,c;\Gamma) \to X(a,b,c) \simeq \PP_\C^1
\end{equation}
(generalizing the $j$-invariant) is a Belyi map, unramified away from $\{0,1,\infty\}$ (by our normalization).  Accordingly, the curve $X(a,b,c;\frakN)$ descends to a number field \cite[Theorem B]{ClarkVoight}.

In light of the motivation above, we focus now on a nice class of congruence subgroups.
The $E$-subalgebra $A \colonequals E \langle \Delta^{(2)} \rangle \leq \M_2(\R)$ generated by (any lift of) the image of $\Delta^{(2)} \hookrightarrow \PSL_2(\R)$ is a quaternion algebra, and $\Lambda \colonequals \Z_E\langle \Delta \rangle$ is a $\Z_E$-order in $A$.  Then there is a commutative square
\begin{equation} \label{eqn:delta2intonormaliz}
\begin{aligned}
\xymatrix{
\Delta^{(2)} \ar@{^(->}[r] \ar@{^(->}[d] & \Lambda^1/\{\pm 1\} \ar@{^(->}[d] \\
\Delta \ar@{^(->}[r] & N_{A^\times}(\Lambda)/E^\times 
} 
\end{aligned}
\end{equation}
where $\Lambda^1 \colonequals \{\gamma \in \Lambda : \nrd(\gamma)=1\}$ are the elements of reduced norm $1$. 

Suppose $\frakN \subseteq \Z_E$ is coprime to $\discrd(\Lambda)$ and $\frakd_{F|E}$, the relative discriminant of $F$ over $E$; for example, this holds if $\frakN$ is coprime to $2abc$.  Then the reduction $\Lambda \to \Lambda/\frakN\Lambda \simeq \M_2(\Z_E/\frakN)$ gives a well-defined group homomorphism 
\begin{equation} \label{eqn:piNN}
\pi_\frakN \colon \Delta \to \PGL_2(\Z_E/\frakN).
\end{equation}
Combined with \eqref{eqn:delta2intonormaliz}, we obtain a commutative diagram (\Cref{prop:sl2andpgl2})
\begin{equation} \label{eqn:delta2pdwe}
\begin{aligned}
\xymatrix{
\Delta^{(2)} \ar@{->}[r] \ar@{^(->}[d] & \SL_2(\Z_E/\frakN)/\{\pm 1\} \ar@{->}[d] \\
\Delta \ar@{->}[r]^(.33){\pi_\frakN} & \PGL_2(\Z_E/\frakN) 
} 
\end{aligned}
\end{equation}
Let $G_\frakN \colonequals \pi_\frakN(\Delta)$ be the image.  Let $s^\sharp$ be the order of $\pi_\frakN(\delta_s)$ for $s=a,b,c$; then $a^\sharp,b^\sharp,c^\sharp$ are the ramification degrees in $\varphi_\frakN$ above $0,1,\infty$, and the homomorphism $\pi_\frakN$ factors through $\Delta(a^\sharp,b^\sharp,c^\sharp)$.  To avoid redundancy, we say that $\frakN$ is \defi{admissible} for $(a,b,c)$ if $s^\sharp=s$ for all $s=a,b,c$---so in particular, $s \neq \infty$.

Without loss of generality (but see \Cref{prop:otherCasesX0}), we now suppose that $\frakN$ is admissible for $(a,b,c)$.  Then the main result of Clark--Voight \cite[Theorem A]{ClarkVoight} (see \Cref{the:ClarkVoight}) describes the group $G_\frakN$.  For example, when $\frakN=\frakp$ is \emph{prime}, then 
\begin{equation} \label{eqn:GNW}
G_\frakN \simeq \PXL_2(\Z_E/\frakp) 
\end{equation}
where $\PXL_2=\PSL_2$ if $\frakp$ (necessarily unramified in $F$) splits completely in $F$, and otherwise $\PXL_2=\PGL_2$.

With this in mind, in this paper we focus on the case where $\frakN=\frakp$ is prime.  This case is already a quite interesting first step, and still relevant for our motivation (see the next section).  Moreover, the case of composite level $\frakN$ builds on the prime level case and at the same time introduces several new challenges that are not present in prime level.  We plan to pursue the general case in future work.

Returning now to our original motivation, the usual upper-triangular (Borel-type) subgroups 
\begin{equation}
H_{1,\frakp} \leq H_{0,\frakp} \colonequals \begin{pmatrix} * & * \\ 0 & * \end{pmatrix} \leq \PGL_2(\Z_E/\frakp)
\end{equation}
naturally include into $G_\frakp$ via \eqref{eqn:GNW}.  We define the Borel-type congruence subgroups of $\Delta$
\begin{equation}
\begin{aligned}
\Gamma_0(a,b,c;\frakp) &\colonequals \pi_\frakp^{-1}(H_{0,\frakp}) \\
\Gamma_1(a,b,c;\frakp) &\colonequals \pi_\frakp^{-1}(H_{1,\frakp}) \\
\end{aligned}
\end{equation}
We write $X_0(a,b,c;\frakp)$ and $X_1(a,b,c;\frakp)$ for the corresponding quotients.  
For $(a,b,c)=(2,3,\infty)$, we recover the classical modular curves $X_0(p)$ and $X_1(p)$.

Our main result is as follows.

\begin{theorem} \label{thm:mainthm}
For any $g \in \Z_{\geq 0}$, there are only finitely many Borel-type triangular modular curves of genus $g$ with admissible prime level $\frakN=\frakp$.  The number of curves of genus at most $2$ are as follows:
\begin{equation*}
\bgroup
\def\arraystretch{1.15}
    \begin{array}{c||c|c|c}
    \multicolumn{1}{c}{} & \multicolumn{3}{c}{\textup{genus}} \\
         &0 &1 &2 \\
    \hline\hline
    X_0(a,b,c;\frakp) & 69 & 248 & 453 \\ \hline
    X_1(a,b,c;\frakp) &6& 9 & 11 
    \end{array}
\egroup
\end{equation*}
\end{theorem}

Our proof of \Cref{thm:mainthm} includes a complete enumeration, computed using an implementation in \textsf{Magma} \cite{Magma} available online \cite{codeTMC} (including the list in computer readable format with additional data).  For the list with $g=0,1$ and prime level, see \Cref{app:tables}.

\subsection*{Discussion}

As for classical modular curves, \Cref{thm:mainthm} uses the Riemann--Hurwitz theorem.  We observe that the ramification at prime level takes a tidy form.  We carry out the explicit enumeration using the existence and classification results of Clark--Voight \cite{ClarkVoight}, which themselves ultimately rest on work of Macbeath \cite{Macbeath} classifying two-generated subgroups of $\SL_2(\F_q)$ in terms of trace triples. The case $a=2$ causes particular difficulties (see \Cref{rem:aEquals2}).

Our theorem has potential applications in arithmetic geometry analogous to classical modular curves.  Just as the quotient of the upper half-plane by $\PSL_2(\Z)$ is the set of complex points of the moduli space of elliptic curves (parametrized by the affine $j$-line), Cohen--Wolfart \cite[\S 3.3]{CohenWolfart} and Archinard \cite{Archinard} showed that the curves $X(a,b,c)$ over $\C$ naturally parametrize \emph{hypergeometric abelian varieties}, certain Prym varieties of cyclic covers of $\PP^1$ branched over $\leq 4$ points.  The name comes from the fact that their complex periods are values of ${}_2F_1$-hypergeometric functions for the parameter $t \in \PP^1(\C) \smallsetminus \{0,1,\infty\}$.  In accordance with Manin's ``unity of mathematics'' \cite{Clemens}, their point counts are defined by finite-field analogues of hypergeometric functions for $t \in \PP^1(\F_q) \smallsetminus \{0,1,\infty\}$;  these can be packaged together (in an $\ell$-adic Galois representation) to define hypergeometric $L$-functions attached to a motive for every $t \in \PP^1(\Q^{\textup{al}}) \smallsetminus \{0,1,\infty\}$. 

More generally, just as classical modular curves parametrize elliptic curves equipped with level structure, triangular modular curves parametrize hypergeometric abelian varieties equipped with level structure: see upcoming work of Kucharczyk--Voight \cite{kv} for the details, including a natural idelic refinement and a notion of canonical model.  In this light, our paper classifies those situations where we might parametrize \emph{infinitely many} such varieties with (nontrivial Borel-type) level structure for $t \in \Q$.  

As shown by Takeuchi \cite{TakeuchiArithmetic,TakeuchiCommesurability}, only finitely many triples $(a,b,c)$ give rise to arithmetic Fuchsian groups; the remaining triples are \emph{nonarithmetic}.  Thus almost all of the corresponding triangular modular curves are \emph{thin} subgroups of the adelic points of a quaternionic group, so fall outside the usual scope of the Langlands program.  

As a final possible Diophantine application, we recall work of Darmon \cite{Darmon14th}: he provides a dictionary between finite index subgroups of the triangle group $\Delta(a,b,c)$ and approaches to solve the generalized Fermat equation $x^a+y^b+z^c=0$.  From this vantage point, the triangular modular curves of low genus ``explain'' situations where the associated mod $\frakp$ Galois representations are reducible.  

In future work, we plan to compute equations for these curves (as Belyi maps) using the methods of Klug--Musty--Schiavone--Voight \cite{KMSV} and then to study their rational points.  Even without these equations, we have verified that all but a handful of the genus zero curves necessarily have a ramified rational point (hence are isomorphic to $\PP^1$ over any field of definition).  It would also be interesting to pursue cases when $\frakp$ ramifies in $A$, where the corresponding Galois covers will instead be solvable.

To conclude, we peek ahead to more general triangular modular curves, allowing other subgroups $\Gamma \leq \Gamma(a,b,c;\frakN)$ (prescribing other possible images of the corresponding Galois representations).  For the case $\Delta=\PSL_2(\Z)$, the story is a long and beautiful one, originating with a conjecture of Rademacher that there are only finitely many genus $0$ congruence subgroups of $\PSL_2(\Z)$.  Thompson \cite{Thompson} proved this for any genus $g$, but the list of Cummins--Pauli relies upon difficult and delicate $p$-adic methods of Cox--Parry \cite{CoxParry} for an explicit bound on the level in terms of the genus.  We propose the following conjecture, which predicts a similar result for triangular modular curves.

\begin{conj} \label{conj:gfintriang}
For all $g \in \Z_{\geq 0}$, there are only finitely many admissible triangular modular curves of genus $g$. 
\end{conj}

We consider our main result (\Cref{thm:mainthm}) as partial progress towards this conjecture---the Borel--type subgroups are the family with the smallest growing index, thus likely to have the smallest genera.  It would be interesting to see if the rather delicate $p$-adic methods of Cox--Parry can be generalized from $\PSL_2(\Z/N\Z)$ to groups of the form $\PXL_2(\Z_E/\frakN)$, as this would imply \Cref{conj:gfintriang} in an effective way.

\subsection*{Contents}
In \cref{sec:curveDef}, we set up triangular modular curves and as a warmup consider the much easier Galois case $X(a,b,c;\frakp)$.  In \cref{sec:triang}, we extend the work of Clark--Voight to understand the arithmetic requirements to construct triangular modular curves.
Then in \cref{sec:X0}, for the case $X_0(a,b,c;\frakp)$ with $a,b,c \in \Z$, we give an explicit formula for the genus and we bound the norm of the level in terms of the genus, proving finiteness; we then provide an algorithm to effectively enumerate them in \cref{sec:enumeration}.  
In \cref{sec:X1}, we provide analogous results for curves $X_1(a,b,c;\frakN)$, and finally we prove \Cref{thm:mainthm}.  We conclude by providing the list in \Cref{app:tables}. 

\subsection*{Acknowledgements}

The authors would like to thank Asher Auel and Robert Kucharczyk for helpful conversations and the anonymous referees for their constructive feedback.  The authors were supported by a Simons Collaboration grant (550029, to Voight).

\section{Setup and definitions} \label{sec:curveDef}

In this section, we give some basic setup and notation, define congruence subgroups, and consider the enumeration problem in the Galois case; for further reference, see Clark--Voight \cite{ClarkVoight}.  

\subsection*{Triangle groups}

Beginning again, let $a,b,c\in\Z_{\ge 2}\cup\{\infty\}$.  Let
\begin{equation}\label{eqn:chi}
    \chi(a,b,c)\colonequals\frac{1}{a}+\frac1b+\frac1c-1
\end{equation}
so that $\chi(a,b,c)\pi$ measures difference from $\pi$ of the sum of the angles of a triangle with angles $\pi/a,\pi/b,\pi/c$.  If $\chi(a,b,c) \geq 0$, then such a triangle is drawn on the sphere or Euclidean plane, and these are very classical.  Otherwise, we $\chi(a,b,c)<0$ and we say that the triple $(a,b,c)$ is \defi{hyperbolic}, as then the triangle lies in the (completed) upper half-plane $\mathcal{H}$.  For a hyperbolic triple $(a,b,c)$, we always have
\begin{equation}\label{eqn:boudChi42}
    \chi(a,b,c)\le \chi(2,3,7)=-\frac{1}{42}
\end{equation}
bounded away from zero, by a simple maximization argument by cases.

As in the introduction, let $\triangle(a,b,c)$ be the subgroup of orientation-preserving isometries of the group generated by reflections in the sides of the triangle described above, drawn in the appropriate geometry.  Then we have a presentation
\begin{equation}
\label{eqn:trianglePresentation}
    \triangle(a,b,c)\colonequals\langle\delta_a,\delta_b,\delta_c\,|\,\delta_a^a=\delta_b^b=\delta_c^c=\delta_a\delta_b\delta_c=1\rangle.
\end{equation}
where $\delta_s$ corresponds to a counterclockwise rotation at the vertex with angle $2\pi/s$.  
    By cyclic permutation and inversion \cite[Remark 2.2]{ClarkVoight}, we can reorganize the generators and suppose without loss of generality that
    \begin{equation} \label{eqn:abcordered}
    a\le b\le c.
    \end{equation}

From now on, we suppose that the triple $(a,b,c)$ is hyperbolic.  Then there is an associated map $\triangle(a,b,c) \hookrightarrow \PSL_2(\R)$, unique up to conjugation.  We will often suppress the dependence on the triple from notation, writing for example $\Delta=\Delta(a,b,c)$.

The group $\Delta$ is said to be \defi{cocompact} if the quotient of the upper half-plane by $\Delta$ is compact, else we say $\Delta$ is \defi{noncocompact}.  We have $\Delta$ noncocompact if and only if at least one of $a,b,c$ is equal to $\infty$.

Let $\Delta^{(2)}$ denote the subgroup of $\Delta$ generated by the set of squares $\{\delta^2 : \delta \in \Delta\}$.  Then $\Delta^{(2)} \trianglelefteq \Delta$ is a normal subgroup, in fact \cite[(5.9)]{ClarkVoight} the quotient $\Delta/\Delta^{(2)}$ is represented by the elements $\delta_s$ with $s \in \{a,b,c\}$ such that either $s=\infty$ or $s \in \Z_{\geq 2}$ is even, hence
\begin{equation} \label{eqn:take124}
\Delta/\Delta^{(2)} \simeq
\begin{cases}
\{0\}, & \text{if at least two of $a,b,c$ are odd integers}; \\
\Z/2\Z, & \text{if exactly one of $a,b,c$ is an odd integer}; \\
(\Z/2\Z)^2, & \text{if all of $a,b,c$ are even integers or $\infty$}.
\end{cases}
\end{equation}

\begin{lemma} \label{lem:Delta2gens}
The group $\Delta^{(2)}$ is generated by the set
\begin{equation} 
\{ \delta_{s}^{-1}\delta_t^2\delta_s : s,t \in \{a,b,c\} \} \cup \{ \delta_s \delta_t \delta_s^{-1} \delta_t^{-1} : s,t \in \{a,b,c\}\}. 
\end{equation}
\end{lemma}

\begin{proof}
Follows from Takeuchi \cite[Lemma 3, Proposition 5]{TakeuchiArithmetic}: the generating set presented there is smaller (depending on cases), whereas we collect these and symmetrize to make a uniform statement.
\end{proof}

\subsection*{Quaternions}

For $s \in \Z_{\geq 2} \cup \{\infty\}$, let $\zeta_s \colonequals \exp(2\pi i/s)$ and let $\lambda_s \colonequals \zeta_s+1/\zeta_s=2\cos(2\pi/s)$, with $\zeta_\infty=1$ and $\lambda_\infty=2$ by convention.  Define the tower of fields
\begin{equation}
\begin{tikzcd}
    F=F(a,b,c)\colonequals\Q(\lambda_{2a},\lambda_{2b},\lambda_{2c}) \arrow[d,dash] \\
    E=E(a,b,c)\colonequals\Q(\lambda_a,\lambda_b,\lambda_c,\lambda_{2a}\lambda_{2b}\lambda_{2c}).
\end{tikzcd}    
\end{equation}
The extension $F \supseteq E$ is abelian of exponent at most $2$ (since $\lambda_{2s}^2 = \lambda_s+2$) and has degree at most $4$.  Let $\Z_F \supseteq \Z_E$ be the corresponding rings of integers, and let $\frakd_{F|E}$ be the relative discriminant of $F \,|\, E$.  The field $F$ is the trace field of the image of $\triangle$ in $\PSL_2(\R)$, and $E$ the trace field for $\Delta^{(2)}$, also called the \defi{invariant trace field} (see Maclachlan--Reid \cite[Section 5.5]{MaclachlanReid}).

As above, we have a map $\Delta \hookrightarrow \PSL_2(\R)$; the 
$F$-subalgebra $B \colonequals F \langle \Delta \rangle \leq \M_2(\R)$ generated by any lift of the image (well-defined, since $-1 \in F$) is a quaternion algebra, similarly $\calO \colonequals \Z_F\langle \Delta \rangle$ is a $\Z_F$-order in $B$ \cite[Propositions 2--3]{Takeuchi00}.  The reduced discriminant of $\calO$ is a principal ideal of $\Z_F$ generated by \cite[Lemma~5.4]{ClarkVoight}
\begin{equation}\label{eqn:beta}
    \beta(a,b,c) \colonequals \lambda_{2a}^2+\lambda_{2b}^2+\lambda_{2c}^2+\lambda_{2a}\lambda_{2b}\lambda_{2c} - 4 = \lambda_a + \lambda_b+\lambda_c + \lambda_{2a}\lambda_{2b}\lambda_{2c}+2 \in \Z_E.
\end{equation}

The same construction applies to $\Delta^{(2)}$, yielding a quaternion $E$-algebra $A$ and a $\Z_E$-order $\Lambda$.  Let $\calO^1 \colonequals \{\gamma \in \calO : \nrd(\gamma)=1\}$ be the elements of reduced norm $1$ in $\calO$, and define $\Lambda^1$ similarly.  Then we have a commutative square of group homomorphisms 
\begin{equation} \label{eqn:delta12}
\begin{aligned}
\xymatrix{
\Delta^{(2)} \ar@{^(->}[r] \ar@{^(->}[d] & \Lambda^1/\{\pm 1\} \ar@{^(->}[d] \\
\Delta \ar@{^(->}[r] & \calO^1/\{\pm 1\} 
} 
\end{aligned}
\end{equation}
In fact, the bottom map descends to the \emph{normalizer} $N_{A}(\Lambda)$ of $\Lambda$ in $A$, as follows.

\begin{lemma} \label{lem:descenttonormalizer}
The composition of the maps
\[ \Delta \hookrightarrow \frac{\calO^1}{\{\pm 1\}} \hookrightarrow \frac{N_{B^\times}(\calO)}{F^\times} \]
factors via the map
\begin{equation} \label{eqn:deltaembed}
\begin{aligned}
\Delta &\hookrightarrow \frac{N_{A^\times}(\Lambda)}{E^\times} \\
\delta_s &\mapsto \begin{cases}
\delta_s^2+1 = \lambda_{2s}\delta_s, &\textup{if $s \neq 2$}; \\
(\delta_c^2+1)(\delta_b^2+1) = \lambda_{2b}\lambda_{2c}\delta_a, & \textup{if $s=a=2$};
\end{cases}
\end{aligned}
\end{equation}
followed by the natural inclusion
$N_{A^\times}(\Lambda)/E^\times \hookrightarrow N_{B^\times}(\calO)/F^\times$. 
\end{lemma}

\begin{proof}
See Clark--Voight \cite[Proposition 5.13]{ClarkVoight}.  (The description fails to be uniform when $a=2$ because $\lambda_4=0$; since $a \leq b \leq c$ we must have $b > 2$, else $(a,b,c)$ is not hyperbolic.  The map is nevertheless uniquely determined, since $\delta_a\delta_b\delta_c=1$.)
\end{proof}

\subsection*{Congruence subgroups: general definition}

We now define congruence subgroups.  Let $\frakN \subseteq \Z_E$ be a nonzero ideal.  Then reducing elements modulo $\frakN$, as in \eqref{eqn:delta12} we obtain a commutative diagram
\begin{equation} \label{eqn:delta12NN}
\begin{aligned}
\xymatrix{
1 \ar[r] & \Gamma^{(2)}(\frakN) \ar[r] \ar@{^(->}[d] & \Delta^{(2)} \ar[r] \ar@{^(->}[d] & (\Lambda/\frakN\Lambda)^1/\{\pm 1\} \ar@{^(->}[d] \\
1 \ar[r] & \Gamma(\frakN) \ar[r] & \Delta \ar[r]^(.3){\varpi_\frakN} & (\calO/\frakN\calO)^1/\{\pm 1\}
} 
\end{aligned}
\end{equation}
but now with kernels in the rows: in particular, we have a group homomorphism
\begin{equation} \label{eqn:homomorphismPSL}
\varpi_\frakN \colon \Delta \to (\calO/\frakN\calO)/\{\pm 1\}
\end{equation}
with kernel
\begin{equation} 
\Gamma(\frakN) \colonequals \ker \pi_\frakN = \{ \delta \in \Delta : \delta \equiv \pm 1 \psmod{\frakN\calO} \} \trianglelefteq \Delta 
\end{equation}
called the \defi{principal congruence subgroup of level $\frakN$}.  As in the introduction, we define \defi{congruence subgroups} of $\Delta$ to be those that contain a principal congruence subgroup, and a \defi{triangular modular curve} to be a quotient of the (completed) upper half-plane by a congruence subgroup of a triangle group, for example
\begin{equation} 
X(\frakN)=X(a,b,c;\frakN) \colonequals \Gamma(\frakN) \backslash \mathcal{H}
\end{equation}
are called the \defi{principal} triangular modular curves.

\begin{remark}
One could work more generally with ideals of $\Z_F$ instead, arriving at the same definition of congruence subgroups but with a different notion of level.  In light of what follows, especially the robust failure of $\varpi_\frakN$ to be surjective, we prefer to work with levels in $\Z_E$.
\end{remark}

Since $\Delta$ normalizes $\Delta^{(2)}$ and therefore $\Lambda$ and $\frakN\Lambda$, there is descent to the normalizer as in \Cref{lem:descenttonormalizer}.  However, the precise description of $\Gamma(\frakN)$ depends on the ramification behavior of the primes dividing $\frakN$ in the extension $F\,|\,E$ and in the algebras $A$ and $B$ (and this already introduces some subtleties when $\frakN$ is composite).  We pursue this in the next section.

\subsection*{Galois case}

Before proceeding, as a warmup we consider the curves $X(a,b,c;\frakp)$ corresponding to principal congruence subgroups, where $X(a,b,c;\frakp) \to X(a,b,c) \simeq \PP^1$ is a generically Galois Belyi map.

Quite generally, for any generically Galois Belyi map with group $G$, the ramification indices above each ramification point are equal.  Without loss of generality, we may suppose that $a,b,c$ are also the orders of the ramification points.  Thus the Riemann-Hurwitz formula gives
\begin{equation}
    2g(X)-2=-2(\#G) + \sum_{s=a,b,c}\frac{\#G}{s}(s-1)
\end{equation}
which simplifies to
\begin{equation}\label{eqn:genusX}
g(X) = 1-\frac{\#G}{2}\chi(a,b,c).
\end{equation}
From this genus formula and \eqref{eqn:boudChi42}, we can conclude that, for any fixed genus $g_0\ge 0$, there are finitely many hyperbolic $G$-Galois Belyi maps with genus $g_0$.  

We are of course interested in the special case where
\[ G=\Gamma(\frakp) \backslash \Delta \simeq \PXL_2(\F_\frakp) \]
where $\F_\frakp \colonequals \Z_E/\frakp$ is the residue field and $\PXL_2(\F_\frakp)$ denotes either $\PSL_2(\F_\frakp)$ or $\PGL_2(\F_\frakp)$.  (The major task in the next section is to precisely investigate this arithmetically.)  Plugging $G=\PXL_2(\F_q)$ into the above:
\begin{equation*}
    84(g_0-1)\ge \#G=q(q+1)(q-1)\cdot \begin{cases} 1/2, &\text{ if $G=\PSL_2(\F_q)$ and $q$ is odd};\\ 
    1, & \text{ otherwise}.
    \end{cases}
\end{equation*}
Thus, there are no curves $X(a,b,c;\frakp)$ of genus at most $1$.
For genus 2, we can use the inequality to see that $q$ must be less than 6, so $\#G\le 60$ and, if $g(X(a,b,c;\frakp))=2$, then
\begin{equation*}
    -\frac{1}{\chi(a,b,c)}\le 30.
\end{equation*}
This inequality implies that $a\le b\le c\le 7$ and, by checking the genera of these possibilities with \eqref{eqn:genusX}, we conclude that there are no curves $X(a,b,c;\frakp)$ of genus 2.  

In fact, the smallest genus for a hyperbolic triple with $a,b,c \in \Z_{\geq 2}$ is genus $3$ for $(a,b,c)=(2,3,7)$, yielding the famed Klein quartic curve.  More generally, see Clark--Voight \cite[Table 10.5]{ClarkVoight} for examples up to genus $24$.

\section{Triangular modular curves} \label{sec:triang}

In this section, we study triangular modular curves generalizing the classical modular curves; the main results are \Cref{prop:sl2andpgl2}, where we define the relevant matrix representation of $\Delta$, and Theorem \ref{the:ClarkVoight}, describing its image building on work of Clark--Voight \cite{ClarkVoight}.  Throughout, we retain our notation from the previous section.

\subsection*{Congruence subgroups: matrix case}

We return to \eqref{eqn:delta12NN}, and identify matrix groups.   Recalling \eqref{eqn:beta}, we first suppose that $\beta = \discrd \calO$ is coprime to $\frakN$, so all primes $\frakp \mid \frakN$ are unramified in $B$ but more strongly we have $(\calO/\frakN\calO)^1/\{\pm 1\} \simeq \SL_2(\Z_F/\frakN\Z_F)/\{\pm 1\}$.  

For the order $\Lambda$, we recall \Cref{lem:Delta2gens}: given $a,b,c$, we can compute its $\Z_E$-module span in $A$ and therefore a $\Z_E$-pseudobasis for $\Lambda$, hence its reduced discriminant. Since $\Lambda\Z_F \subseteq \calO$, we have $\beta \mid \discrd(\Lambda)$ \cite[Corollary 5.17]{ClarkVoight}.

So we make the stronger assumption that $\frakN$ is coprime to $\discrd(\Lambda)$.  Then from \eqref{eqn:delta12} we get
\begin{equation} \label{eqn:delta12withmats}
\begin{aligned}
\xymatrix{
\Delta^{(2)} \ar[r] \ar@{^(->}[d] & (\Lambda/\frakN\Lambda)^1/\{\pm 1\} \ar@{^(->}[d] \ar[r]^(0.475){\sim} & \SL_2(\Z_E/\frakN)/\{\pm 1\}\\
\Delta \ar[r]^(0.3){\pi_\frakN} & (\calO/\frakN\calO)^1/\{\pm 1\}  \ar[r]^(0.45){\sim} & \SL_2(\Z_F/\frakN\Z_F)/\{\pm 1\}
} 
\end{aligned}
\end{equation}

To descend the bottom map to the normalizer as in \Cref{lem:descenttonormalizer}, we restrict our scope taking $\frakN=\frakp$ prime and work just a little bit more.  

Let
\begin{equation}
\Z_{E,(\frakp)} \colonequals \{\alpha \in E : \ord_\frakp(\alpha) \geq 0\} \subseteq E
\end{equation}
be the localization of $\Z_E$ at the ideal $\frakp$ (all elements coprime to $\frakp$ become units).

\begin{lemma} \label{lem:uupstoit}
Suppose that $\frakp \nmid \frakd_{F|E}$.  Then for $s=a,b,c$, we can write
\begin{equation} \label{eqn:uups}
\lambda_s+2 = \upsilon_s \theta_s^2 \in E^\times 
\end{equation}
with:
\begin{itemize}
\item $\upsilon_s \in \Z_{E,(\frakp)}^\times$, well-defined up to multiplication by an element of $\Z_{E,(\frakp)}^{\times 2}$, i.e., up to the square of an element of $\Z_{E,(\frakp)}^{\times}$, and 
\item $\theta_s \in E^\times$, well-defined up to $\Z_{E,(\frakp)}^{\times}$.  
\end{itemize}
If $\frakp$ is coprime to $2abc$, then we may take $\theta_s=1$ and $\upsilon_s=\lambda_s+2$.

Moreover, the prime $\frakp$ (necessarily unramified in $F$) splits completely in $F$ if and only if the Kronecker symbols $(\upsilon_s\,|\,\frakp)=1$ are trivial for all $s=a,b,c$.
\end{lemma}

\begin{proof}
First, a bit of generality: for $\alpha \in E^\times$ with even valuation at all primes $\frakp \mid \frakN$, by weak approximation in $E$ we can write
\begin{equation} \label{eqn:uups0}
\alpha = \upsilon \theta^2 \in E^\times 
\end{equation}
with $\upsilon,\theta$ as in the statement of the lemma.  

Now to apply this, we observe that $F=E(\lambda_{2a},\lambda_{2b},\lambda_{2c})$ and recall that $\lambda_{2s}^2=\lambda_s+2$.  By hypothesis, we have $\frakp \nmid \frakd_{F|E}$; in particular the elements $\lambda_s+2$ must have even (nonnegative) valuation at $\frakp$.  Thus \eqref{eqn:uups0} applies, giving \eqref{eqn:uups}.  The final statement follows from the usual splitting criterion in quadratic fields.
\end{proof}

We obtain the following result.

\begin{proposition} \label{prop:sl2andpgl2}
Suppose that $\frakp \nmid \discrd(\Lambda)\frakd_{F|E}$.  Then there is a commutative diagram
\begin{equation} \label{eqn:delta2pdwe0}
\begin{aligned}
\xymatrix{
\Delta^{(2)} \ar@{->}[r] \ar@{^(->}[d] & \SL_2(\Z_E/\frakp)/\{\pm 1\} \ar@{->}[d] \\
\Delta \ar@{->}[r]^(.33){\pi_\frakN} & \PGL_2(\Z_E/\frakp) 
} 
\end{aligned}
\end{equation}
and the map $\pi_\frakN \colon \Delta \to \PGL_2(\Z_E/\frakN)$ factors through $\varpi_\frakN$.
\end{proposition}

We let $G_\frakp \colonequals \pi_\frakp(\Delta) \leq \PGL_2(\Z_E/\frakp)$ be the image of $\pi_\frakp$.  

\begin{proof}
Combine \eqref{eqn:delta12withmats} with \Cref{lem:uupstoit}.
\end{proof}

\begin{remark}
A similar argument works when $\frakN$ is composite; however the right-hand vertical map $\SL_2(\Z_E/\frakN)/\{\pm 1\} \to \PGL_2(\Z_E/\frakN)$ may no longer be injective when $\frakN$ is composite.  This leads to certain ambiguities about the definition which we will return to in future work.
\end{remark}

\subsection*{Admissibility, projectivity, and image}

It can and does happen that two different triangular modular curves are isomorphic (as curves and as covers of $\PP^1$).  The issue is simply that in the homomorphism $\pi_\frakN$ from $\Delta(a,b,c)$ to a matrix group, the generators $\delta_s$ need not have order $s$ in the image (for $s=a,b,c$).  In other words, the reduction homomorphism factors through a triangle group with a smaller triple.  This happens for example when $s=\infty$, as the order of $\pi_\frakN(\delta_s)$ is always finite!  To illustrate this phenomena, we present the following example.

\begin{example} \label{exm:uhoh}
Consider the triples $(2,3,c)$ with $c=p^k$, where $k \geq 1$ and $p \geq 5$ is prime.  Then 
\begin{equation} E_k \colonequals E(2,3,c)=F(2,3,c)=\Q(\lambda_{2c})=\Q(\zeta_{2c})^+ 
\end{equation}
and $\beta(2,3,c) = \lambda_{c}-1 \in \Z_{E_k}^\times$.  The prime $p$ is totally ramified in $F$, so $\F_{\frakp_k} \simeq \F_p$ for $\frakp_k \mid p$.  Thus $X(2,3,p^k;\frakp_k) \simeq X(2,3,p;\frakp_1)$. 
\end{example}

To avoid this redundancy, we make the following definition.  

\begin{definition}\label{def:admissible}
Given a triple $(a,b,c)$, an ideal $\frakp \subseteq \Z_{E(a,b,c)}$ is \defi{admissible} for $(a,b,c)$ if 
\begin{itemize}
\item $\frakp \nmid \discrd(\Lambda)\frakd_{F|E}$, and
\item the order of $\pi_\frakp(\delta_s)$ is equal to $s$ for all $s=a,b,c$.
\end{itemize}
\end{definition}

\begin{theorem}[Clark--Voight] \label{the:ClarkVoight}
We have $\pi_\frakp(G_\frakp^{(2)})=\PSL_2(\Z_E/\frakp)$ and 
\[ \pi_\frakp(G_\frakp)=\PXL_2(\Z_E/\frakp) \]
where $\PXL_2$ denotes $\PSL_2$ or $\PGL_2$ according as $\frakp$ splits in $F \supseteq E$ or not.  
\end{theorem}

\begin{proof}
We refer to Clark--Voight \cite[Theorem A]{ClarkVoight} for the case where $\frakp \nmid 2abc$; but examining the argument given \cite[Remark 5.24, proof of Theorem 9.1]{ClarkVoight} in light of the above, we see that it extends when $\frakp \nmid \discrd(\Lambda)\beta\frakd_{F|E}$.  
\end{proof}

\subsection*{Hyperbolic triples reducing to non-hyperbolic triples}

In considering admissible triples, we may lose the hypothesis that $(a,b,c)$ is hyperbolic; however, this situation is easy to characterize.  We note that in most cases, these groups do not contain $\PSL_2(\F_q)$, so they are not considered in this paper.

\begin{proposition}\label{prop:otherCasesX0}
Let $(a,b,c)\in(\Z_{\ge0}\cup\{\infty\})^3$ be a triple and let $\frakp\subset\Z_E$ be a nonzero prime ideal.  Suppose further that $(a^\sharp,b^\sharp,c^\sharp)$ is a projective triple, but not hyperbolic.  Then $(a,b,c;p,q)$ is one of the elements listed in the following table.  In the table, $p$ lies below $\frakp$ and $q$ is the residue field degree of $\frakp$.
\begin{equation}
\bgroup
\def\arraystretch{1.15}
\begin{array}{c|c||c|c|c||c}\label{eqn:tableExtras}
    \bm{(a,b,c)} & \text{\textbf{conditions}} &\;\;\bm{p} \;\;&\;\;\bm{q} \;\;&\bm{{\rm PXL}}&\bm{E(a^\sharp,b^\sharp,c^\sharp)}\\\hline\hline
    \begin{array}{c}
    (2^{k_a},2^{k_b},3\cdot 2^{k_c}),\\
    (3\cdot 2^{k_c},\infty,\infty),\\
    (2^{k_a},3\cdot 2^{k_c},\infty)
     \end{array}&1\le k_a< k_b&2&2&1&\href{https://www.lmfdb.org/NumberField/1.1.1.1}{\Q}
     \\\hline
    \begin{array}{c}
    (3^{k_a},3^{k_b},3^{k_c}),\\
    (3^{k_a},\infty,\infty),\\
    (3^{k_a},3^{k_b},\infty),\\
    (\infty,\infty,\infty)
     \end{array}
     & 1\le k_a\le k_b<k_c &3&3 & 1 & \href{https://www.lmfdb.org/NumberField/1.1.1.1}{\Q}\\\hline
    \begin{array}{c}
    (2\cdot 3^{k_a},3^{k_b},3^{k_c}),\\
    (2\cdot 3^{k_a} ,3^{k_b},\infty),\\
    (2\cdot 3^{k_a},\infty,\infty)
    \end{array}
    &1\le k_b\le k_c,\,k_ak_bk_c\ne 1 &3&3 &1& \href{https://www.lmfdb.org/NumberField/1.1.1.1}{\Q}\\\hline
    \begin{array}{c}
    (2\cdot 3^{k_a},3^{k_b},4\cdot 3^{k_c}),\\
    (2\cdot 3^{k_a},4\cdot 3^{k_b},\infty)
    \end{array}
    &1\le k_b,\,k_ak_bk_c\ne 1&3&3&-1&\href{https://www.lmfdb.org/NumberField/1.1.1.1}{\Q}\\\hline
    \begin{array}{c}
    (2^{k_a},3\cdot 2^{k_b},5 \cdot 2^{k_c}),\\(3\cdot 2^{k_b},5 \cdot 2^{k_c},\infty)\end{array}
    &1\le k_a,\,k_ak_bk_c\ne 1&2&4&1&\href{https://www.lmfdb.org/NumberField/2.2.5.1}{\Q(\sqrt{5})}\\\hline
    \begin{array}{c}
    (2\cdot 5^{k_a},3\cdot 5^{k_b},5^{k_c}),\\(2\cdot 5^{k_a},3\cdot 5^{k_b},\infty)\end{array}&1\le k_c,\,k_ak_bk_c\ne 1&5&5&1&\href{https://www.lmfdb.org/NumberField/2.2.5.1}{\Q(\sqrt{5})}
\end{array}
\egroup
\end{equation}
Furthermore, the curves $X(a,b,c;\frakp)$ with $(a,b,c;\frakp)$ as above all have genus 0.
\end{proposition}

\begin{proof}
We first focus on the prime ideal case and make a case by case study.  The only triples $(a,b,c)\in(\Z_{\ge0}\cup\{\infty\})^3$ that are not hyperbolic are $$(2,2,n)\text{ for }n > 1,\,(2,3,3),\, (2,3,4),\, (2,3,5),\, (2,3,6),\,(2,4,4),\text{ or }(3,3,3).$$ 

Assume first that $(a^\sharp,b^\sharp,c^\sharp)=(2,2,c)$ for $c>1$. The image of $\pi_\frakp:\Delta(2,2,c)\to \PGL_2(\F_q)$ must be dihedral.  The only dihedral group that is isomorphic to $\PXL_2(\F_q)$ for any $q$ is $D_6\simeq\PSL_2(\F_2)$. Thus, we only have the triple $(a^\sharp,b^\sharp,c^\sharp)=(2,2,3)$ and prime $\frakp_2$.

The group $\Delta(2,3,6)$ is solvable since it fits in the exact sequence:
\begin{equation*}
    1\to \Z^2\to \Delta(2,3,6)\to \Z/6\Z\to 1.
\end{equation*}
The only solvable groups of the form $\PXL_2(\F_q)$ are $S_4 \simeq \PGL_2(\F_3)$ and $A_4 \simeq \PSL_2(\F_3)$. The triple $(2,3,6)$ is not admissible for $q=2$ or $q=3$, so $(2,3,6)$ is not projective and admissible for any prime ideal $\frakp$. With the same analysis, we can rule out $(2,4,4)$. We also have that the group $\Delta(3,3,3)$ is solvable. Hence, the image of $\pi_\frakp:\Delta(3,3,3)\to\PXL_2(\F_q)$ must be solvable.  The only solvable groups of this form are $A_4 \simeq \PSL_2(\F_3)$ and $S_4 \simeq \PGL_2(\F_3)$.  Thus, the only option is that $\frakp$ is a prime above $3$ with residue field $\F_3$.

The last triples to consider are $(2,3,3),\,(2,3,4)$ and $(2,3,5)$. These triples are all \emph{exceptional}. Tthe only projective linear groups that can arise from exceptional triples \cite[Remark 8.4]{ClarkVoight} are the following:
$$\PSL_2(\F_3), \PGL_2(\F_3),\PGL_2(\F_4), \PSL_2(\F_5).$$
We now use this fact to finish the analysis. When $(a^\sharp,b^\sharp,c^\sharp)=(2,3,3)$, the admissible prime ideals $\frakp$ have residue field degree $3,4$ and $5$. The field $E(2,3,3)$ is the rational field, so $\Z_E/\frakp_2\simeq\F_2$.
In addition, the ideal $2\Z_E$ is totally ramified in any field $E(2\cdot 2^{k_a},3\cdot 2^{k_b},3\cdot 2^{k_c})$, so $q\ne 4$. The only options then are $q=3$ and $q=5$. A quick \textsf{Magma} \cite{Magma} calculation shows that elements with these orders cannot generate $\PSL_2(\F_5)$.

Similarly, when $(a^\sharp,b^\sharp,c^\sharp)=(2,3,4)$, the only possibilities for $q$ which make the triple projective and admissible for $\frakp$ are $q=3$ or $q=5$. However, the field $E(2,3,4)$ is the rational field and 5 is inert in $F$, so we would have $G_{5\Z_E}\simeq \PGL_2(\F_5)$, which is not on the list of possible groups. The same happens for $(a^\sharp,b^\sharp,c^\sharp)=(2,3,5)$; the options of $q$ for an admissible prime $\frakp$ are $q=2,3,4,5$. The ideal $2\Z_E$ is inert in $E(2,3,5)$, an extension of $\Q$ of degree 2, thus $q=2$ is not possible. The ideal $3\Z_E$ is also inert in $E(2,3,5)$, so an isomorphism with $\PXL_2(\F_3)$ is not possible. The only options for $q$ are $q=4$ and $q=5$.

For all of the possible triples $(a^\sharp,b^\sharp,c^\sharp)$ and primes $\frakp$ described above, we certify that such map is possible by exhibiting passports for each curve.  Finally, we use \Cref{eqn:genusX} to compute the genus of each of these curves, finding that they all have genus 0.
\end{proof}

\subsection*{Borel-type subgroups} 

As in the introduction, let 
\begin{equation} \label{eqn:abd}
\left\{\begin{pmatrix} a & b \\ 0 & d \end{pmatrix} : a,b,d \in \Z_E/\frakp \text{ and } ad \in (\Z_E/\frakp)^\times \right\} \leq \GL_2(\Z_E/\frakp) 
\end{equation}
be the upper-triangular matrices in $\GL_2(\Z_E/\frakp)$, and let $H_{0,\frakp}$ be its image in the projection to $\PGL_2(\Z_E/\frakp)$.  Similarly, let
\begin{equation} \label{eqn:1b1}
\left\{\begin{pmatrix} 1 & b \\ 0 & 1 \end{pmatrix} : b \in \Z_E/\frakp \right\} \leq  \GL_2(\Z_E/\frakp) 
\end{equation}
be the upper unipotent subgroup and $H_{1,\frakp}$ again its image in $\PGL_2(\Z_E/\frakp)$.  

We then define the subgroups
\begin{equation} \label{eqn:gammagammgama}
\begin{aligned}
    \Gamma_0(a,b,c;\frakp) &\colonequals\varphi^{-1}_\frakp(H_{0,\frakp}), \\ \Gamma_1(a,b,c;\frakp) &\colonequals\varphi^{-1}_\frakp(H_{1,\frakp}).
\end{aligned}
\end{equation}
and the corresponding quotients
\begin{equation}
\begin{aligned}\label{eqn:X0X1def}
    X_0(a,b,c;\frakp) &\colonequals \Gamma_0(a,b,c;\frakp)\backslash \calH = H_{0,\frakp} \backslash X'(a,b,c;\frakp) \\
    X_1(a,b,c;\frakp) &\colonequals \Gamma_1(a,b,c;\frakp)\backslash \calH = H_{1,\frakp} \backslash X(a,b,c;\frakp).
\end{aligned}
\end{equation}
Then we have natural quotient maps
\begin{equation}\label{eqn:covers}
    X(a,b,c;\frakN) \to X_1(a,b,c;\frakN)\to X_0(a,b,c;\frakN)\to X(a,b,c;1) \simeq \PP^1.
\end{equation}

\section{Triangular modular curves \texorpdfstring{$X_0(a,b,c;\frakp)$}{X0(a,b,c;p)} of prime level}\label{sec:X0}

In this section, we exhibit a formula for the genus of the triangular modular curves $X_0(a,b,c;\frakp)$ for $\frakp$ prime. Using this formula we show that there are only finitely many such curves with bounded genus. 
\subsection*{Setup}

Let $(a,b,c)$ be a hyperbolic triple and $\frakp$ be an admissible prime of $E=E(a,b,c)$ with residue field $\F_\frakp$. Let $q\colonequals\#\F_\frakp$, so $\F_\frakp \simeq \F_q$.  Because $E$ is Galois over $\Q$, all primes $\frakp$ have the same ramification and splitting type; it follows that the genus of $X_0(a,b,c;\frakp)$ only depends on the prime number $p \in \Z$ below $\frakp$ (and the inertial degree of $\frakp$ over $p$).

Let $G\colonequals G_\frakp$ be as in \Cref{the:ClarkVoight}. Then the group $H_0=H_{0,\frakp}$ consists of the image in $G$ of the upper-triangular matrices of $\SL_2(\F_q)$ or $\GL_2(\F_q)$, depending on $G$.  By construction, the curves $X_0(a,b,c;\frakp)$ and $X(a,b,c;\frakp)$ fit in the following diagram.

\begin{equation*}
\begin{aligned}
\xymatrix{
X(a,b,c;\frakp) \ar@{->}[dr]^{H_0} \ar@{->}[dd]_{G}\\
&X_0(a,b,c;\frakp) \ar@{->}[dl] \\
\P^1 &
} 
\end{aligned}
\end{equation*}

We first compute the index $[G:H_0]$, which corresponds to the degree of the cover $X_0(a,b,c;\frakp)\to \P^1$.  If $G=\PGL_2(\F_q)$, up to multiplication by a scalar matrix, it is possible to choose representatives of elements of $H_0$ that have 1 on the first entry of the matrix. Thus, $\# H_0=q(q-1)$ and $[G:H_0]=q+1.$
When $q$ is even, we have an isomorphism $\PSL_2(\F_q) \simeq \PGL_2(\F_q)$, so the index $[G:H_0]$ is the same as above.  Finally, if $G=\PSL_2(\F_q)$ with $q$ odd, then representatives can be chosen to have 1 on the first entry of the matrix as above. Also, the upper triangular matrices are defined up to multiplication by $-1$. Hence $\#H_0=\frac12q(q-1)$ and $[G:H_0]=q+1.$

Via the projection of the first column of the matrix to $\P^1(\F_q)$, the set of cosets $G/H_0$ is naturally in bijection with $\P^1(\F_q)$.  With this bijection, the action of $\pi_\frakp(\Delta)$ on $G/H_0$ becomes simply matrix multiplication.  The ramification of the cover $X_0(a,b,c;p)\to \P^1$ then depends on the cycle decomposition of the corresponding elements (in $G$) as an element of $\Sym(\PP^1) \simeq S_{q+1}$. 

\subsection*{Cycle structure and genus formula}

The following lemma describes the cycle structure using only the order of the elements.  Recall we write $\PXL_2$ for either $\PSL_2$ or $\PGL_2$.

\begin{lemma}\label{lem:cyclesDescription}
Let $G=\PXL_2(\F_q)$ with $q=p^r$ for a prime number $p$.  Let $\overline{\sigma}_s \in G$ have order $s\ge2$, and if $s=2$ suppose $p=2$.  Then the action of $\overline{\sigma}_s$ on $\P^1(\F_q)$ has:
\begin{enumroman}
    \item two fixed points and $(q-1)/s$ orbits of length $s$ if $s\mid (q-1)$; \label{case:s|q-1}
    \item one fixed point and $q/p$ orbits of length $p$ if $s=p$ (this is the case when $s\mid q$); and \label{case:s|p}
    \item (no fixed points and) $(q+1)/s$ orbits of length $s$ if $s \mid (q+1)$.\label{case:s|q+1}
\end{enumroman}
\end{lemma}

\begin{proof}
We note that each class in $G$ is represented by matrices that are diagonalizable over $\F_q$, diagonalizable only over $\F_{q^2}$, or not diagonalizable.  We prove the Lemma by studying in detail each case. Let $\sigma_s$ be an element of $\GL_2(\F_q)$ whose projection to $G$ is $\overline{\sigma}_s$. If $\sigma_s$ is diagonalizable, then we say that $\overline{\sigma}_s$ is \defi{split semisimple}, and $\sigma_s$ is conjugate to say the diagonal matrix $\begin{pmatrix} u & 0 \\ 0 & v\end{pmatrix}$.  We must have $u\neq v$ because otherwise $\overline{\sigma}_s$ would be the identity in $G$, contradicting that $s\ge2$. The order of $\overline{\sigma}_s$ is $s$, so $s$ is the order of $uv^{-1}$ in $\F_q^\times$. To find the orbits of the action of $\overline{\sigma}_s$ on $\P^1(\F_q)$, we use that
$$\begin{pmatrix}
u&0\\0&v
\end{pmatrix}\Pone{1}{0}=\Pone{1}{0},\hspace{0.7cm}\begin{pmatrix}
u&0\\0&v
\end{pmatrix}\Pone{x}{1}=\Pone{uv^{-1}x}{1},$$
for any $x\in\F_q$. Hence, the action of $\overline{\sigma}_s$ has two fixed points: $\Pone{1}{0}$ and $\Pone{0}{1}$, and $(q-1)/s$ orbits with $s$ elements.

The element $\overline{\sigma}_s$ is \defi{unipotent} if and only if it is conjugate to $\begin{pmatrix} 1&u\\0&1\end{pmatrix}$ in $G$ for some $u\in \F_q^\times$. This is the case when the characteristic polynomial of $\sigma_s$ has two equal roots and $\sigma_s$ is not diagonalizable over $\F_q^2$. This happens if and only if $s=p$. In this case, we have
$$\begin{pmatrix} 1&u\\0&1\end{pmatrix}\Pone{1}{0}=\Pone{1}{0},\hspace{0.7cm}\begin{pmatrix} 1&u\\0&1\end{pmatrix}\Pone{x}{1}=\Pone{x+u}{1},$$
where $x\in\F_q$.
There is only one fixed point and there are $q/p$ orbits of size $p$. 

If the characteristic polynomial of $\sigma_s$ does not split in $\F_q$, we call $\overline{\sigma}_s$ \defi{non-split semisimple}. The action of $\overline{\sigma}_s$ has no fixed points because this would imply that $\sigma_s$ has an eigenvector. The splitting field of the characteristic polynomial of $\sigma_s$ is $\F_{q^2}$. Let $\alpha_1,\alpha_2\in\F_{q^2}\setminus\F_q$ be the roots of this polynomial. Then $\sigma_s$ is conjugate with the diagonal matrix $[\alpha_1,\alpha_2]$ with $\sigma_s=T^{-1}[\alpha_1,\alpha_2]T$ for some invertible matrix $T$. For all $m\in\N$ such that $\overline{\sigma}_s^m$ fixes $(x:y)^t\in\P^1(\F_q)$, we have that
$$\begin{pmatrix}
\alpha_1^m&0\\0&\alpha_2^{m}
\end{pmatrix}\left(T\Pone{x}{y}\right)=\left(T\Pone{x}{y}\right).$$
From the analysis of the split semisimple case, we conclude that every orbit has length $s$. Thus, the action of $\sigma_s$ on $\P^1(\F_q)$ has $(q+1)/s$ orbits of length $s$.
\end{proof}

The previous lemma does not consider the case when $s=2$ and $q$ is odd.  The ambiguity arises since if $s=2$ then $s\mid (q-1)$ and $s\mid (q+1)$, so $\overline{\sigma}_2$ can be either split or non-split (semisimple).

\begin{example}\label{rem:aEquals2}
    For $(a,b,c)=(2,3,8)$ and $G=\PGL_2(\F_7)$, we have $\sigma_2$ split.  
    On the other hand, for $(2,6,6)$ and $G=\PGL_2(\F_7)$, we have $\sigma_2$ non-split.
\end{example}

 The following lemma partially solves this problem. 

\begin{lemma}\label{lem:2cyclesDescriprtionPSL}
    Let $G=\PSL_2(\F_q)$ with $q$ odd, and let $\overline{\sigma}_2\in G$ be an element of order $2$. Then the action of $\overline{\sigma}_2$ on $\P^1(\F_q)$ has:
    \begin{enumroman}
        \item two fixed points and $(q-1)/2$ orbits of size $2$ if $-1$ is a square modulo $q$; and
        \item (no fixed points and) $(q+1)/2$ orbits of size $2$, otherwise.
    \end{enumroman}
\end{lemma}
\begin{proof}
    Let $\overline{\sigma}_2$ be a matrix of order 2 in $\PSL_2(\F_q)$.  Pick a lift $\sigma_2\in\SL_2(\F_q)$ of $\overline{\sigma}_2$. Because $\sigma_2^4$ is the identity, its characteristic polynomial must be a quadratic polynomial dividing $x^4-1$. In addition, the constant of this polynomial must be 1 since this is the determinant of $\sigma_2$. The only possibility for such a polynomial is $x^2+1$.  
    If $-1 \in \F_q^{\times 2}$, then this characteristic polynomial splits with distinct roots, so we are in the split semisimple case of \Cref{lem:cyclesDescription}.  Otherwise, $-1$ is not a square and we are in the non-split semisimple case.  
\end{proof}

Now we are ready to give a formula for the genus $g$ of $X_0(a,b,c;p)$.  For $x \in \R$, we write $\lfloor x \rceil$ for the rounding down of $x$, so $\lfloor 3/2 \rceil = 1$.

\begin{theorem} \label{the:genusX0}
Let $(a,b,c)$ be a hyperbolic admissible triple and $\frakp$ be a prime of $E$ above a rational prime $p$.  Then the genus of $X_0(a,b,c;\frakp)$ is given by
\begin{equation}\label{eqn:genusX0}
    g(X_0(a,b,c;\frakp))=-q+\frac{1}{2}\sum_{s\in\{a,b,c\}}\left\lfloor \frac{q}{s} \right\rceil (s-1) + \epsilon(a,b,c;\frakp)
\end{equation}
where $q \colonequals \Nm(\frakp)$ and $\epsilon(a,b,c;\frakp) \in \{0,1/2\}$ is uniquely determined by  $g(X_0(a,b,c;\frakp)) \in \Z$.  Moreover, we have $\epsilon(a,b,c;\frakp)=0$ unless $a=2$ and $q$ is odd.
\end{theorem}

In the latter case ($a=2$ and $q$ odd), \Cref{lem:2cyclesDescriprtionPSL} implies that when $G=\PSL_2(\F_q)$, we have $\epsilon(a,b,c;\frakp)=0$ if and only if $q \equiv 1 \pmod{4}$ (case (i)).  

\begin{proof}
    Consider elements $\overline{\sigma}_a,\,\overline{\sigma}_b,\,\overline{\sigma}_c\in \PXL_2(\F_q)$ of orders $a^,\,b,$ and $c$, respectively, such that $\sigma_a\sigma_b\sigma_c=1$. We recall that the map $X_0(a,b,c;\frakp)\to X(1)$ has degree $q+1$ since $[G:H_0]=q+1$. The Riemann--Hurwitz formula implies
\begin{equation}\label{eqn:riemannHurwitzX0}
    2g-2=-2(q+1)+\epsilon_a+\epsilon_b+\epsilon_c,
\end{equation}
where $\epsilon_s$ is the ramification index at the points that ramify. We can compute $\epsilon_s$ from \Cref{lem:cyclesDescription} and \Cref{lem:2cyclesDescriprtionPSL}, with $\epsilon_s=k_s(s-1)$, where
\begin{equation} \label{eqn:ks}
k_s=\begin{cases} (q-1)/s, &\text{ if $s\mid (q-1)$};\\
q/s,& \text{ if $s \mid q$};\\
(q+1)/s& \text{ if $s \mid (q+1)$};
\end{cases}
\end{equation}
if $s\neq 2$ or ($s=a=2$ and $q$ is even); whereas if $s=a=2$ and $q$ is odd, then either $k_2=(q+1)/2$ or $k_2=(q-1)/2$ is determined by the fact that $g \in \Z$, since they differ by $1$.
\end{proof}

\begin{remark}
    Instead of using parity, in the $\PGL_2(\F_q)$ and $q$ odd case, we can always explicitly compute elements $\overline{\sigma}_2$, $\overline{\sigma}_b$, $\overline{\sigma}_c \in G$, of orders $2$, $b$, and $c$ respectively, such that $\overline{\sigma}_2\overline{\sigma}_b\overline{\sigma}_c=1$. We can then decide if $\overline{\sigma}_2$ is split or non-split and use \Cref{lem:cyclesDescription} to compute the ramification. 
\end{remark}
\subsection*{Algorithm}

We present an implementation of \Cref{the:genusX0}. 

\begin{algorithm}\label{alg:findGenus} 
Let $(a,b,c)$ be a hyperbolic triple and let $\frakp \subseteq \Z_{E(a,b,c)}$ be a nonzero prime ideal. This algorithm computes the genus of $X_0(a,b,c;\frakp)$ and the Galois group $G_\frakp$ of the cover $X(a,b,c;\frakp)\to \P^1$.

\begin{enumalg}
    \item Compute the residue field of $\frakp$ and set $q\colonequals\#\F_\frakp$. 
    \item Compute the residue field $\Z_F/\frakp_F$, where $\frakp_F$ is a prime of $F(a,b,c)$ above $\frakp$. If $\F_q\simeq \Z_F/\frakp_F$, then $G= \PSL_2(\F_q)$. Otherwise set $G=\PGL_2(\F_q)$.
    \item Compute $g$ using \Cref{the:genusX0}.
\end{enumalg}
\end{algorithm}

\begin{proof}[Proof of correctness]
Correctness follows from the formula in \Cref{the:genusX0}.  Steps 1 and 2 can be performed by constructing the algebraic number field; it can also be done purely in terms of the prime number $p$ below $\frakp$ as in \Cref{alg:isSplit}.
\end{proof}

\subsection*{Bounding the genus}
Our goal remains to show that, for fixed genus $g_0$, there are finitely many admissible curves $X_0(a,b,c;\frakp)$ of genus $g\le g_0$. We first characterize the hyperbolic triples $(a,b,c)$ such that the curve $X(a,b,c)$ has Galois group $\PXL_2(\F_q)$, for a given $q$.

In the prime case, the notion of admissible ideal can be turned around, as follows.

\begin{definition}\label{def:qadmissible}
Let $q \colonequals p^r$ be a power of a prime number $p$. A hyperbolic triple $(a,b,c)$ is \defi{$q$-admissible} if $s$ divides at least one integer in the set $\{q-1,\,p,\,q+1\}$
for all $s\in\{a,b,c\}$, not including $\infty$.
\end{definition}

\begin{lemma}\label{lem:possibleOrders}
For any triangular modular curve $X_0(a,b,c;\frakp)$ with $q \colonequals \Nm \frakp$ and $\frakp$ admissible for $(a,b,c)$, the triple $(a,b,c)$ is $q$-admissible.
\end{lemma}

\begin{proof}
As shown in the proof of \Cref{lem:cyclesDescription}, the order of every element in $\PXL_2(\F_q)$ needs to divide one of $\{q-1,p,q+1\}$.
\end{proof}

\begin{proposition}\label{prop:boundOng}
Let $g$ be the genus of the triangular modular curve $X_0(a,b,c;\frakp)$ and set $q\colonequals \Z_E/ \frakp$. Then,
\begin{equation*}
    q\le 84(g+1)+1.
\end{equation*}
\end{proposition}
\begin{proof}
    Let $s^\sharp$ be the order of $\pi_\frakp(\delta_s)$.  The cases where $(a^\sharp,b^\sharp,c^\sharp)$ is not hyperbolic are handled in \Cref{prop:otherCasesX0}: we get $g=0$, and the inequality holds.  So we may suppose without loss of generality that $s^\sharp=s$ for $s=\{a,b,c\}$, and still that $(a,b,c)$ is hyperbolic.
    
    We study the Belyi map $X_0(a,b,c;\frakp)\to \P^1$. Let $\epsilon_a,\epsilon_b,\epsilon_c$ be the ramification degrees of this map. Using \Cref{lem:cyclesDescription}, we have that for $s\in\{a,b,c\}$,
    \begin{equation}\label{eqn:boundsEi}
        (q-1)-\frac{q-1}{s}=\frac{(s-1)(q-1)}{s}\le \epsilon_s\le \frac{(s-1)(q+1)}{s}=(q+1)-\frac{q+1}{s}.
    \end{equation}
    Because of these bounds and \eqref{eqn:riemannHurwitzX0},
\begin{equation}
\begin{aligned}
        g(X_0(a,b,c;\frakp))&\ge -(q+1)+\frac{(a-1)(q-1)}{2a}+\frac{(b-1)(q-1)}{2b}+\frac{(c-1)(q-1)}{2c}+1  \\
        &=(q-1)\left(-1+\frac{3}{2}-\frac{1}{2a}-\frac{1}{2b}-\frac{1}{2c}\right)-1\\
        &=\frac{q-1}{2}\left|\chi(a,b,c)\right|-1,
    \end{aligned}
    \end{equation}
    where $\chi(a,b,c)$ is as in \eqref{eqn:chi}.
    The result then follows from the previous inequality and \eqref{eqn:boudChi42}.
\end{proof}

\begin{corollary}\label{cor:X0ListBounded}
For a fixed genus $g_0\in\Z_{\ge0}$, there are only finitely many hyperbolic triples $(a,b,c)$ and admissible primes $\frakp$ such that the curves $X_0(a,b,c;\frakp)$ have genus $g\le g_0$.
\end{corollary}
\begin{proof}
    By \Cref{prop:boundOng}, we obtain an upper bound on the rational prime $p$ given by $q\le 84(g_0+1)+1$. Also, for $(a,b,c)$ to be $q$-admissible, necessarily $s\le q+1$ for all $s\in\{a,b,c\}$.  This leaves only finitely many possibilities.
\end{proof}

\begin{remark}
To make computations more efficient, we can consider a bound on $q$ that depends on $\chi(a,b,c)$.  For the genus of $X_0(a,b,c;\frakp)$ to be less than or equal to $g_0$, it is necessary that 
\begin{equation}\label{eqn:boundq}
    q\le \frac{2(g_0+1)}{\left|\chi(a,b,c)\right|}+1.
\end{equation}
This inequality also shows that
\begin{equation}\label{eqn:boundChi}
    0< \left|\chi(a,b,c)\right|\le \frac{2(g_0+1)}{q-1}.
\end{equation}
Therefore, we can bound $a$, $b$, and $c$ whenever $q$ is fixed. 
\end{remark}

\section{Enumerating curves of low genus}\label{sec:enumeration}

We present the main algorithms that use the theory developed in \cref{sec:X0}. The goal of this section is to effectively enumerate the curves $X_0(a,b,c;\frakp)$ of bounded genus. The number of curves is finite from \Cref{cor:X0ListBounded}. As explained in \cref{sec:curveDef}, if $\frakp$ is admissible, then $G$ is given by $\PXL_2(\F_q)$.  The first condition (coprimality) in admissibility can be expensive to check, so we first check the easier necessary (but not sufficient) condition that $\frakp \nmid \beta(a,b,c)$.  

\begin{algorithm}\label{alg:isSplit} 
Let $(a,b,c)$ be a hyperbolic triple and $p$ be a prime number. This algorithm returns \verb|true| if there exists a prime $\frakp \subseteq \Z_{E(a,b,c)}$ above $p$ such that $\frakp \nmid \beta(a,b,c)$.  

\begin{enumalg}
    \item If $p\nmid 2abc$, then return \verb|true|.
    \item Find $\F_\frakp=\F_q$, where $\frakp$ is any prime of $E$ above $p$.
    \item Set $m \colonequals \lcm(a,b,c)$. Construct $\F_q(\zeta_{2m})$. Set $z \colonequals \zeta_{2m}$.
    \item For every $i\in(\Z/2m\Z)^\times$, and set $l_{2s}\colonequals z^{im/s}+1/z^{im/s}$ for $s\in\{a,b,c\}$. Compute
    $$\beta_i \colonequals l_{2a}^2 + l_{2b}^2 + l_{2c}^2 + l_{2a}l_{2b}l_{2c} - 4.$$  
    If $\beta_i \neq 0$ and whenever $p \mid s$ we have $s=p$, then return \verb|true|.  Otherwise, return \verb|false|.
\end{enumalg}
\end{algorithm}

\begin{proof}[Proof of correctness]
    Let $\frakp$ be a prime of $\Z_{E(a,b,c)}$ above $p$.
    If $p \nmid 2abc$ then $\frakp \nmid \beta(a,b,c)$ \cite[Lemma 5.5]{ClarkVoight}. When $\frakp \mid abc$, checking that $\frakp$ does not divide $\beta(a,b,c)$ is more involved. We do this in steps 2 to 4 by computing $\beta$ in the residue field of $\frakp$. This computation is independent of the prime $\frakp$ chosen above $p$ because $E$ is Galois over $\Q$.
\end{proof}

Now we are ready to present the main algorithm that ties the results of \cref{sec:X0} into an explicit enumeration.
\begin{algorithm}\label{alg:enumerationX0}
    Returns a list \verb|lowGenus| of all hyperbolic triples $(a,b,c)\in\Z_{\ge 2}^3$ and norms of prime ideals $\frakp$ of $E(a,b,c)$ that are admissible such that the genus of $X_0(a,b,c;\frakp)$ is at most $g_0$.
    \begin{enumalg}
        \item Loop over the list of possible powers $q=p^r$, where $p$ is a prime number and $q\le 84(g_0+1)+1$.
        \item For each $q$ from step 1, find all $q$-admissible hyperbolic triples $(a,b,c)$ (as in \Cref{def:qadmissible}).
        \item For each $q$-admissible triple $(a,b,c)$ from step 2, check if $\chi(a,b,c)$ satisfies \eqref{eqn:boundChi} and if $\frakp$ does not divide $\beta(a,b,c)$ using \Cref{alg:isSplit}. If yes, compute the candidate genus $g$ of $X_0(a,b,c;\frakp)$ using \Cref{alg:findGenus}. 
        \item If $g\le g_0$, check that $\frakp \nmid \discrd(\Lambda)\frakd_{F|E}$.  If yes, add $(a,b,c;q)$ to the list \verb|lowGenus|.
    \end{enumalg}
\end{algorithm}

\begin{proof}[Proof of correctness]
For step 1, see \Cref{prop:boundOng}.  
    Every hyperbolic $q$-admissible triple gives rise to one such curve. The correctness of the rest of the algorithm follows from the work done in \cref{sec:X0}.
\end{proof}

    We list the CPU time (in seconds) for our implementation to compute the list of curves $X_0(a,b,c;\frakp)$ of genus up to bounds 0, 1, and 2 on a standard laptop:
    \begin{equation*}
\bgroup
\def\arraystretch{1.15}
    \begin{array}{c||c|c|c}
        \text{Genus bound} &0&1 &2\\
    \hline\hline
    \text{Time (s)}& 1.7& 9.7 & 1110.3
    \end{array}
    \egroup
\end{equation*}

\section{Triangular modular curves \texorpdfstring{$X_1(a,b,c;\frakp)$}{X1'(a,b,c;p)}}\label{sec:X1}

In this section, we use  \cref{sec:X0} to give analogous results for triangular modular curves $X_1(a,b,c;\frakp)$, completing the proof of our main result.  

We recall that $X_1(a,b,c;\frakp)$ is defined in \eqref{eqn:X0X1def} as the quotient of $\calH$ by $\Gamma_1(a,b,c;\frakp)$. 

\begin{corollary}\label{cor:finiteX1}
    For any integer $g_0\ge0$, there are finitely many triangular modular curves $X_1(a,b,c;\frakp)$
    with $\frakp$ admissible.
\end{corollary}
\begin{proof}
For every triple $(a,b,c)\in(\Z_{\ge2}\cup\{\infty\})^3$ and prime ideal $\frakp$ of $E(a,b,c)$, there is a cover $X_1(a,b,c;\frakp)\to X_0(a,b,c;\frakp)$.  All curves $X_1(a,b,c;\frakp)$ of genus bounded above by $g_0$ cover curves $X_0(a,b,c;\frakp)$ of genus bounded above by $g_0$.  Because of \Cref{cor:X0ListBounded}, there are finitely many admissible triples $(a,b,c)$ and prime ideals $\frakp$ that give rise to curves $X_0(a,b,c;\frakp)$ of genus bounded above by $g_0$.
\end{proof}

We now focus on explicitly enumerating all curves of bounded genus.  The goal first is to prove group-theoretic results that describe the degree and ramification of the cover $X_1(\frakp)\to X(1)$.  We describe the structure of the quotient $\PXL_2(\F_q)$ modulo $H_{1,\frakp}$ and then describe the action of $\pi_\frakp(\delta_s)$ on this quotient.  The main difference with \cref{sec:X0} is that the quotient $G/H_{0,\frakp}$ does not depend on $G$ being isomorphic to $\PSL_2(\F_q)$ or $\PGL_2(\F_q)$, whereas the structure of $G/H_{1,\frakp}$ depends on the choice of $G$. Let $H_1\colonequals H_{1,\frakp}$.

\begin{lemma}\label{lem:quotientByH1}
    Let $G=\PXL_2(\F_q)$, where $\F_q\colonequals\Z_E/\frakp$. The quotient $G/H_1$ can be described as follows.
    \begin{enumroman}
        \item If $G=\PSL_2(\F_q)$, then $G/H_1\simeq (\F_q\times \F_q\setminus\{(0,0)\})/\langle\pm1\rangle$: explicitly, the class of $(x,z)\in \F_q\times \F_q$ maps to the coset of $\begin{pmatrix} x&y\\z&w\end{pmatrix}$, where $y,w \in \F_q$ satisfy $xw-yz=1$.
        \item If $G=\PGL_2(\F_q)$, then $G/H_1\simeq (\F_q\times \F_q\setminus\{(0,0)\})/\langle\pm1\rangle\times \F_q^\times/\F_q^{\times2}$: explicitly, for $\mu\in \F_q^{\times} \smallsetminus \F_q^{\times 2}$ the class of $((x,y),u)\in ( \F_q\times \F_q\setminus\{(0,0)\})\times \F_q^\times$ maps to the coset of $\begin{pmatrix} x&y\\z&w\end{pmatrix}$, where $z,w \in \F_q$ satisfy $xw-yz=1$ if $u$ is a square and $xw-yz=\mu$ otherwise.
    \end{enumroman}
\end{lemma}
\begin{proof}
    Let $G=\PSL_2(\F_q)$ with $q$ odd. 
    Because $\# H_1= \#\F_q$, we have that $[G:H_1]=(q^2-1)/2.$
    The coset representatives of $G/H_1$ can be parameterized by $(x,z)\in (\F_q\times \F_q)/\langle \pm 1\rangle$.  Indeed, two elements in $\PSL_2(\F_q)$ are in the same coset of $G/H_1$ if and only if there is $\alpha\in\F_q$ such that
    $$\pm\begin{pmatrix}x&y\\z&w\end{pmatrix} \begin{pmatrix}1&\alpha\\0&1\end{pmatrix}=\pm\begin{pmatrix} x&x\alpha+y\\z&z\alpha+w\end{pmatrix}=\begin{pmatrix}x'&y'\\z'&w'\end{pmatrix},$$
    which is the case if and only if $(x,z)=\pm(x',z').$ Thus, the map $(\F_q\times \F_q)/\langle \pm 1\rangle\to G/H_1$ defined by the parametrization is a well-defined, injective homomorphism.  By a cardinality comparison it follows that it is an isomorphism.
    
    Now we let $G=\PGL_2(\F_q)$, 
    so $[G:H_1]=q^2-1$. We claim that the quotient $G/H_1$ is isomorphic to $(\F_q\times\F_q\setminus\{(0,0)\})/\{\pm1\}\times\F_q^\times/\F_q^{\times2}$. To present this isomorphism, we fix a non-square $\mu\in\F_q^\times \smallsetminus \F_q^{\times 2}$. For any $\pm(x,z)\in (\F_q\times\F_q\setminus\{(0,0)\})/\langle\pm1\rangle$, and any $u\in\{1,\mu\}\simeq\F_q^\times/\F_q^{\times2}$, we choose values of $y,w\in\F_q$ such that $xw-yz=u$ and map $\pm(x,z)$ to the class of the matrix $\begin{pmatrix}x&y\\z&w\end{pmatrix}$ in $\PGL_2(\F_q)$. Given two different choices $y,w\in\F_q$ and $y',w'\in\F_q$, if $x\neq 0$, then
    $$\pm\begin{pmatrix}x&y\\z&w\end{pmatrix}=\begin{pmatrix}x&y'\\z&w'\end{pmatrix}\begin{pmatrix}1&x^{-1}(y-y')\\0&1\end{pmatrix}.$$
    If $x=0$, then $z\neq0$ and $0\neq u=yz=y'z$. Thus, $y=y'$. Also, 
    $$\pm\begin{pmatrix}0&y\\z&w\end{pmatrix}=\begin{pmatrix}0&y\\z&w'\end{pmatrix}\begin{pmatrix}1&z^{-1}(w-w')\\0&1\end{pmatrix}.$$
    Thus, the map $(\F_q\times\F_q\setminus\{(0,0)\})/\{\pm1\}\times\F_q^\times/\F_q^{\times2}\to G/H_1$ is a well defined homomorphism. In addition, multiplication by elements in $H_1$ does not change the square class of the determinant or the first column of the matrix, so the homomorphism described above is injective. Since the cardinalities of the domain and range are equal, we conclude that this is an isomorphism.
\end{proof}

We proceed to describe the ramification of the cover $X_1(a,b,c;\frakp)\to\P^1$.  This lemma is similar to \Cref{lem:cyclesDescription}.  The main difference is that in certain cases there are more fixed points than strictly necessary.

\begin{lemma}\label{lem:cycleDecompositionX1}  
    Let $\overline{\sigma}_s\in G=\PXL_2(\F_q)$ and assume that the order of $\overline{\sigma}_s$ is $s$. The structure of the action of $\overline{\sigma}_s$ on $G/H_1$ is as follows:
    \begin{enumroman}
        \item if $\sigma_s$ is semisimple, then there are (no fixed points and) $\frac{[G:H_1]}{s}$ orbits of length $s$,
        \item if $\sigma_s$ is unipotent, then:
        \begin{enumalphii}
            \item if $G=\PSL_2(\F_q)$ and $q$ is odd, there are $(q-1)/2$ fixed points and $(q^2-q)/(2p)$ orbits of length $p$,
            \item otherwise, there are $q-1$ fixed points and $(q^2-q)/p$ orbits of length $p$.
        \end{enumalphii}
    \end{enumroman}
\end{lemma}
 
\begin{proof}
    We use the description of the quotient $G/H_1$ given in \Cref{lem:quotientByH1}.  Let $\sigma_s$ be any element of $\GL_2(\F_q)$ that maps to $\overline{\sigma}_s$ in the quotient to $G$.
    
    If $\sigma_s$ is split semisimple, then it is conjugate over $\F_q$ to a diagonal matrix with entries $u,v$. Because the order of $\overline{\sigma}_s$ is $s$, then $s$ is the order of $uv^{-1}$.  We pick a class in the quotient $G/H_1$ represented by a matrix $M$.  If the class of $M$ is fixed by the action of $\overline{\sigma}_s$, then the first column of $M$ is, up to multiplication by $\pm1$, fixed by multiplication by the diagonal matrix. This implies that $(u,v)=\pm(1,1)$, contradicting that $s\ge2$. Thus, there are no fixed points of the action of $\overline{\sigma}_s$ on $G/H_1$. A similar argument shows that orbits of elements that are not fixed cannot have length less than $s$. Thus, every element belongs to an orbit of length $s$.
    
    If $\sigma_s$ is non-split semisimple, then $\sigma_s$ is split in a quadratic extension of $\F_q$. We assume that $\sigma_s=T^{-1}[\alpha_1,\alpha_2]T$ in this extension. If $\sigma_s^r$ fixes an element for $r\ge1$, then we have
    $$\pm \begin{pmatrix}\alpha_1^r&0\\0&\alpha_2^{r}\end{pmatrix}T\begin{pmatrix}x&y\\z&w\end{pmatrix}=T\begin{pmatrix}x&y'\\z&w'\end{pmatrix}.$$ 
    Multiplication by $T$ does not change the equality in $G/H_1$. Thus, we are back to the split semisimple case and the orbits of the action of $\overline{\sigma}_s$ all have size $s$.
    
    If $\sigma_s$ is unipotent, then $\sigma_s$ can be chosen (by multiplying by scalar matrices) to be conjugate to an upper diagonal matrix with ones in the diagonal. Then,
    $$\begin{pmatrix}1&u\\0&1\end{pmatrix}\begin{pmatrix}x&y\\z&w\end{pmatrix}=\begin{pmatrix}x+uz&y+uw\\z&w\end{pmatrix},$$
    so the class of this matrix in $G/H_1$ is fixed by multiplication by $\sigma_s$ if and only if $uz=0$. Since $s\ge2$, then $z$ must be $0$. We note that if $z\neq0$, then the orbit of the element has length $p$. In $G=\PSL_2(\F_q)$ there are $(q-1)/2$ representatives for which $z=0$, i.e. fixed points. Similarly, if $G=\PGL_2(\F_q)$, then there are $q-1$ fixed points.
\end{proof}

\begin{corollary}\label{cor:genusX1}
Let $(a,b,c)\in\Z_{\ge2}^3$ be a $q$-admissible hyperbolic triple. Let $\frakp$ be a prime ideal of $E(a,b,c)$ above a rational prime $p$. Then the genus of $X_1(a,b,c;\frakp)$ is given by
\begin{equation*}
    g(X_1(a,b,c;\frakp))=-[G:H_1]+1+\frac{1}{2}\sum_{s\in\{a,b,c\}}k_s(s-1),
\end{equation*}
where
\begin{equation*}
    k_s=\begin{cases} 
    (q^2-q)/(2p), & \text{ if $s = p$ and $G=\PSL_2(\F_q)$};\\
    (q^2-q)/p, & \text{ if $s=p$ and $G=\PGL_2(\F_q)$};\\
    (q^2-1)/s, & \text{ if $s\ne p$  and $G=\PGL_2(\F_q)$}; and\\
    (q^2-1)/(2s), & \text{ if $s\ne p$ and $q$ is odd and $G=\PSL_2(\F_q)$}.
    \end{cases}
\end{equation*}
\end{corollary}
\begin{proof}
This formula is given by using the Riemann-Hurwitz formula on $X_1(\frakp)\to \P^1$ and \Cref{lem:cycleDecompositionX1}.
\end{proof}

Now we are ready to present an algorithm that enumerates all curves $X_1(a,b,c;\frakp)$.
\begin{algorithm} \label{alg:lowgenusX1}
    Returns a list \verb|lowGenusX1|  of all hyperbolic triples $(a,b,c)$ and admissible ideals $\frakp$ such that the genus of $X_1(a,b,c;\frakp)$ is $g\le g_0$.
    \begin{enumalg}
        \item Loop over all hyperbolic triples $(a,b,c)$ and prime ideals $\frakp$ such that $X_0(a,b,c;\frakp)$ has genus bounded above by $g_0$.  This list can be obtained from \Cref{alg:enumerationX0}.
        \item For each triple $(a,b,c)$ and ideal $\frakp$ of the previous step, compute the genus $g$ of $X_1(a,b,c;\frakp)$ with \Cref{cor:genusX1}. If $g\le g_0$, then add $(a,b,c;\frakp)$ to the list \verb|lowGenusX1|.
    \end{enumalg}
\end{algorithm}
\begin{proof}[Proof of correctness]
    For all triples $(a,b,c)$ and prime ideals $\frakp$ there are maps $X_1(a,b,c;\frakp)\to X_0(a,b,c;\frakp)$. Thus, the only curves $X_1(a,b,c;\frakp)$ that can have genus $g\le g_0$ must be covering curves $X_0(a,b,c;\frakp)$ of genus bounded above by $g_0$.
\end{proof}

\subsection*{Proof of theorem}

We conclude the paper by proving our main result.

\begin{proof}[Proof of \textup{\Cref{thm:mainthm}}]
By \Cref{cor:X0ListBounded}, there are only finitely many curves $X_0(a,b,c;\frakp)$ with nontrivial admissible prime level $\frakp$ and genus $g \leq g_0$.  Since every curve $X_1(a,b,c;\frakp)$ covers $X_0(a,b,c;\frakp)$, the same is true for $X_1(a,b,c;\frakp)$ (see \Cref{cor:finiteX1}).

For the computation, we run \Cref{alg:enumerationX0} with $g_0=2$, adding extra cases according to \Cref{prop:otherCasesX0}.  To finish, we run \Cref{alg:lowgenusX1}.  The implementation of this computation can be found in our \textsf{Magma} code \cite{codeTMC}.
\end{proof}

\vfill\newpage
\appendix

\section{Tables} \label{app:tables}
We present tables of all hyperbolic triples $(a,b,c)$ and admissible primes $\frakp$ such that the curve $X_0(a,b,c;\frakp)$ has genus 0 or 1.  The list with additional data is available online \cite{codeTMC}.

To record $\frakp$, we list the prime number $p$ below $\frakp$. 
We describe the group $G=\PXL_2(\F_q)$ by presenting $q$ and writing $1$ in the $\PXL$ field if $G=\PSL_2(\F_q)$ and $-1$ if $G=\PGL_2(\F_q)$.  We also record the information about the field $E(a,b,c)$ and the number of different prime ideals of $E$ above $p$.

Nugent--Voight \cite{NugentVoight} define an invariant, the \defi{arithmetic dimension} $\adim(a,b,c)$, to be the dimension of a quaternionic Shimura variety attached to $\Delta(a,b,c)$ given by the number of split real places of $E(a,b,c)$ of the quaternion algebra $A = E\langle \Delta^{(2)} \rangle$.  In particular, the triangle group $\Delta(a,b,c)$ is arithmetic if and only if $\adim(a,b,c)=1$.  

One subtlety is that there can be an isomorphism between the cover coming from a nonarithmetic group and the cover coming from an arithmetic group.  This can only happen when the arithmetic group is of noncompact type, with 
\begin{align*} 
(a,b,c)&=(2,3,\infty),(2,4,\infty),(2,6,\infty),(2,\infty,\infty),(3,3,\infty),(3,\infty,\infty), \\
&\qquad (4,4,\infty),(6,6,\infty),(\infty,\infty,\infty)
\end{align*}
by Takeuchi \cite{TakeuchiCommesurability}.  All of these arise from finite-index subgroups of $\PSL_2(\Z)$, so they are related to classical modular curves, and are defined over $\Q$.  The ramification of the curve $X_0(a,b,c;\frakp)$ for $a,b,c \in \Z_{\geq 2} \cup \{\infty\}$ replaces any occurrence of $\infty$ by $p$; this allows one to readily identify when this extra isomorphism applies.  We record this by adding $(1)$ to the arithmetic dimension entry on the table.

For the arithmetic triangle groups $\Delta(a,b,c)$ such that $\Delta\simeq\Lambda^1$, the corresponding list of curves is contained in \cite[Tables 4.1--4.7]{Voight}.  We confirmed that the intersection is in agreement.

Finally, for noncocompact triples see \Cref{prop:otherCasesX0}.

\newpage

\subsection*{Genus 0, \texorpdfstring{$X_0(a,b,c;\frakp)$}{X0(a,b,c;p)}}
\begin{footnotesize}
\begin{equation*}
\bgroup
\def\arraystretch{1.15}
\begin{array}{c|c||c|c||c|c|c}
    \bm{(a,b,c)}& \bm{p} &\bm{q}&\bm{{\rm PXL}} & \bm{{\rm adim}} &\bm{ E(a,b,c)}&\# \text{ of }\frakp
    \\\hline\hline(2, 3, 7) & 7 & 7 & 1 & 1 & \href{https://www.lmfdb.org/NumberField/3.3.49.1}{\Q(\lambda_7)}&1
    \\\hline(2, 3, 7) & 2 & 8 & 1 & 1 & \href{https://www.lmfdb.org/NumberField/3.3.49.1}{\Q(\lambda_7)}&1
    \\\hline(2, 3, 7) & 13& 13& 1& 1& \href{https://www.lmfdb.org/NumberField/3.3.49.1}{\Q(\lambda_7)}&3
    \\\hline(2, 3, 7) & 29& 29& 1& 1& \href{https://www.lmfdb.org/NumberField/3.3.49.1}{\Q(\lambda_7)}&3
    \\\hline(2, 3, 7) & 43& 43& 1& 1& \href{https://www.lmfdb.org/NumberField/3.3.49.1}{\Q(\lambda_7)}&3
    \\\hline(2, 3, 8) & 7& 7& -1& 1& \href{https://www.lmfdb.org/NumberField/2.2.8.1}{\Q(\sqrt{8})}&2
    \\\hline(2, 3, 8) & 3 & 9 & -1 & 1 &\href{https://www.lmfdb.org/NumberField/2.2.8.1}{\Q(\sqrt{8})}&1
    \\\hline(2, 3, 8) & 17& 17& 1& 1& \href{https://www.lmfdb.org/NumberField/2.2.8.1}{\Q(\sqrt{8})}&2
    \\\hline(2, 3, 8) & 5 & 25 & -1 & 1 &\href{https://www.lmfdb.org/NumberField/2.2.8.1}{\Q(\sqrt{8})}&1
    \\\hline(2, 3, 9) & 19& 19& 1& 1&\href{https://www.lmfdb.org/NumberField/3.3.81.1}{\Q(\lambda_9)}&3
    \\\hline(2, 3, 9) & 37& 37& 1& 1& \href{https://www.lmfdb.org/NumberField/3.3.81.1}{\Q(\lambda_9)}&3
    \\\hline(2, 3, 10) & 11& 11& -1& 1& \href{https://www.lmfdb.org/NumberField/2.2.5.1}{\Q(\sqrt{5})}&2
    \\\hline(2, 3, 10) & 31& 31& -1& 1& \href{https://www.lmfdb.org/NumberField/2.2.5.1}{\Q(\sqrt{5})}&2
    \\\hline(2, 3, 12) & 13 & 13 & -1 & 1 & \href{https://www.lmfdb.org/NumberField/2.2.12.1}{\Q(\sqrt{12})}&2
    \\\hline(2, 3, 12) & 5& 25& 1& 1& \href{https://www.lmfdb.org/NumberField/2.2.12.1}{\Q(\sqrt{12})}&1
    \\\hline(2, 3, 13) & 13 & 13 & 1 & 2\,(1) & \href{https://www.lmfdb.org/NumberField/6.6.371293.1}{\Q(\lambda_{13})}&1
    \\\hline(2, 3, 15) & 2 & 16 & 1 & 2\,(1) & \href{https://www.lmfdb.org/NumberField/4.4.1125.1}{\Q(\lambda_{15})}&1
    \\\hline(2, 3, 18) & 19& 19& -1& 1& \href{https://www.lmfdb.org/NumberField/3.3.81.1}{\Q(\lambda_9)}&3
    \\\hline(2, 4, 5) & 5 & 5 & -1 & 1 & \href{https://www.lmfdb.org/NumberField/2.2.5.1}{\Q(\sqrt{5})}&1
    \\\hline(2, 4, 5) & 3 & 9 & 1 & 1 & \href{https://www.lmfdb.org/NumberField/2.2.5.1}{\Q(\sqrt{5})}&1
    \\\hline(2, 4, 5) & 11& 11& -1& 1& \href{https://www.lmfdb.org/NumberField/2.2.5.1}{\Q(\sqrt{5})}&2
    \\\hline(2, 4, 5) & 41& 41& 1& 1& \href{https://www.lmfdb.org/NumberField/2.2.5.1}{\Q(\sqrt{5})}&2
    \\\hline(2, 4, 6) & 5& 5& -1& 1& \href{https://www.lmfdb.org/NumberField/1.1.1.1}{\Q}&1
    \\\hline(2, 4, 6) & 7& 7& -1& 1& \href{https://www.lmfdb.org/NumberField/1.1.1.1}{\Q}&1
    \\\hline(2, 4, 6) & 13 & 13 & -1 & 1 & \href{https://www.lmfdb.org/NumberField/1.1.1.1}{\Q}&1
    \\\hline(2, 4, 8) & 3 & 9 & -1 & 1 & \href{https://www.lmfdb.org/NumberField/2.2.8.1}{\Q(\sqrt{8})}&1
    \\\hline(2, 4, 8) & 17& 17& 1& 1& \href{https://www.lmfdb.org/NumberField/2.2.8.1}{\Q(\sqrt{8})}&2
    \\\hline(2, 4, 12) & 13& 13& -1& 1& \href{https://www.lmfdb.org/NumberField/2.2.12.1}{\Q(\sqrt{12})}&2
    \\\hline(2,5,5) & 5 & 5 & 1 & 1 & \href{https://www.lmfdb.org/NumberField/2.2.5.1}{\Q(\sqrt{5})} &1
    \\\hline(2,5,5) & 11 & 11 & 1 & 1 & \href{https://www.lmfdb.org/NumberField/2.2.5.1}{\Q(\sqrt{5})} &2
    \\\hline(2, 5, 10) & 11& 11& -1& 1& \href{https://www.lmfdb.org/NumberField/2.2.5.1}{\Q(\sqrt{5})}&2
    \\\hline(2, 6, 6) & 7 & 7 & -1 & 1 & \href{https://www.lmfdb.org/NumberField/1.1.1.1}{\Q}&1
    \\\hline(2, 6, 6) & 13& 13& 1& 1&\href{https://www.lmfdb.org/NumberField/1.1.1.1}{\Q}&1
    \\\hline(2, 6, 7) & 7 & 7 & -1 & 2\,(1)& \href{https://www.lmfdb.org/NumberField/3.3.49.1}{\Q(\lambda_7)}&1
    \\\hline(2, 8, 8) & 3& 9& -1& 1& \href{https://www.lmfdb.org/NumberField/2.2.8.1}{\Q(\sqrt{8})}&1
    \\\hline(3, 3, 4) & 7& 7& 1& 1& \href{https://www.lmfdb.org/NumberField/2.2.8.1}{\Q(\sqrt{8})}&2
    \\\hline(3, 3, 4) & 3 & 9 & 1 & 1 & \href{https://www.lmfdb.org/NumberField/2.2.8.1}{\Q(\sqrt{8})}&1
    \\\hline(3, 3, 4) & 5& 25&1& 1& \href{https://www.lmfdb.org/NumberField/2.2.8.1}{\Q(\sqrt{8})}&1
    \\\hline(3, 3, 5) & 2 & 4 & 1 & 1 & \href{https://www.lmfdb.org/NumberField/2.2.5.1}{\Q(\sqrt{5})}&1
    \\\hline(3, 3, 6) & 13& 13& 1& 1& \href{https://www.lmfdb.org/NumberField/2.2.12.1}{\Q(\sqrt{12})}&2
    \\\hline(3,4,4) & 5 & 5 & -1 & 1 & \href{https://www.lmfdb.org/NumberField/1.1.1.1}{\Q} &1
    \\\hline(3,4,4) & 13 & 13 & -1 & 1 & \href{https://www.lmfdb.org/NumberField/2.2.5.1}{\Q} &1
    \\\hline(3,6,6) & 7 & 7 & -1 & 1 & \href{https://www.lmfdb.org/NumberField/1.1.1.1}{\Q} &1
    \\\hline(4, 4, 4) & 3& 9& 1& 1& \href{https://www.lmfdb.org/NumberField/2.2.8.1}{\Q(\sqrt{8})}&1
    \end{array}
    \egroup
\end{equation*}
\end{footnotesize}

\subsection*{Genus 1, \texorpdfstring{$X_0(a,b,c;\frakp)$}{X0(a,b,c;p)}}
\begin{footnotesize}
This long table is split into three tables (over the next three pages).

\begin{equation*}
\bgroup
\def\arraystretch{1.15}
\begin{array}{c|c||c|c||c|c|c}
    \bm{(a,b,c)}& \bm{p} &\bm{q}&\bm{{\rm PXL}} &\bm{{\rm adim}}&\bm{ E}&\#\text{ of }\frakp
    \\\hline\hline(2, 3, 7) & 3 & 27 & 1 & 1 & \href{https://www.lmfdb.org/NumberField/3.3.49.1}{\Q(\lambda_7)}&1
    \\\hline(2, 3, 7) & 41 & 41 & 1 & 1 &  \href{https://www.lmfdb.org/NumberField/3.3.49.1}{\Q(\lambda_7)}&3
    \\\hline(2, 3, 7) & 71 & 71 & 1 & 1 & \href{https://www.lmfdb.org/NumberField/3.3.49.1}{\Q(\lambda_7)}&3
    \\\hline(2, 3, 7) & 97 & 97 & 1 & 1 & \href{https://www.lmfdb.org/NumberField/3.3.49.1}{\Q(\lambda_7)}&3
    \\\hline(2, 3, 7) & 113 & 113 & 1 & 1 & \href{https://www.lmfdb.org/NumberField/3.3.49.1}{\Q(\lambda_7)}&3
    \\\hline(2, 3, 7) & 127 & 127 & 1 & 1 & \href{https://www.lmfdb.org/NumberField/3.3.49.1}{\Q(\lambda_7)}&3
    \\\hline(2, 3, 8) & 23 & 23 & -1 & 1 & \href{https://www.lmfdb.org/NumberField/2.2.8.1}{\Q(\sqrt{8})}&2
    \\\hline(2, 3, 8) & 31 & 31 & 1 & 1 &\href{https://www.lmfdb.org/NumberField/2.2.8.1}{\Q(\sqrt{8})}&2
    \\\hline(2, 3, 8) & 41 & 41 & -1 & 1 & \href{https://www.lmfdb.org/NumberField/2.2.8.1}{\Q(\sqrt{8})}&2
    \\\hline(2, 3, 8) & 73 & 73 & -1 & 1 & \href{https://www.lmfdb.org/NumberField/2.2.8.1}{\Q(\sqrt{8})}&2
    \\\hline(2, 3, 8) & 97 & 97 & 1 & 1 & \href{https://www.lmfdb.org/NumberField/2.2.8.1}{\Q(\sqrt{8})}&2
    \\\hline(2, 3, 9) & 2 & 8 & 1 & 1 &\href{https://www.lmfdb.org/NumberField/3.3.81.1}{\Q(\lambda_9)}&1
    \\\hline(2, 3, 9) & 17 & 17 & 1 & 1 & \href{https://www.lmfdb.org/NumberField/3.3.81.1}{\Q(\lambda_9)}&3
    \\\hline(2, 3, 9) & 73 & 73 & 1 & 1 & \href{https://www.lmfdb.org/NumberField/3.3.81.1}{\Q(\lambda_9)}&3
    \\\hline(2, 3, 10) & 3 & 9 & -1 & 1&\href{https://www.lmfdb.org/NumberField/2.2.5.1}{\Q(\sqrt{5})}&1
    \\\hline(2, 3, 10) & 19 & 19 & 1 & 1 & \href{https://www.lmfdb.org/NumberField/2.2.5.1}{\Q(\sqrt{5})}&2
    \\\hline(2, 3, 10) & 41 & 41 & 1 & 1 & \href{https://www.lmfdb.org/NumberField/2.2.5.1}{\Q(\sqrt{5})}&2
    \\\hline(2, 3, 10) & 61 & 61 & 1 & 1 & \href{https://www.lmfdb.org/NumberField/2.2.5.1}{\Q(\sqrt{5})}&2
    \\\hline(2, 3, 11) & 11 & 11 & 1 & 1 & \href{https://www.lmfdb.org/NumberField/5.5.14641.1}{\Q(\lambda_5)}&1
    \\\hline(2, 3, 11) & 23 & 23 & 1 & 1 &\href{https://www.lmfdb.org/NumberField/5.5.14641.1}{\Q(\lambda_5)}&5
    \\\hline(2, 3, 12) & 11 & 11 & -1 & 1 &\href{https://www.lmfdb.org/NumberField/2.2.12.1}{\Q(\sqrt{12})}&2
    \\\hline(2, 3, 12) & 37 & 37 & -1 & 1 & \href{https://www.lmfdb.org/NumberField/2.2.12.1}{\Q(\sqrt{12})}&2
    \\\hline(2, 3, 12) & 7 & 49 & 1 & 1 & \href{https://www.lmfdb.org/NumberField/2.2.12.1}{\Q(\sqrt{12})}&1
    \\\hline(2, 3, 13) & 5 & 25 & 1 & 2 & \href{https://www.lmfdb.org/NumberField/6.6.371293.1}{\Q(\lambda_{13})}&3
    \\\hline(2, 3, 13) & 3 & 27 & 1 & 2 & \href{https://www.lmfdb.org/NumberField/6.6.371293.1}{\Q(\lambda_{13})}&2
    \\\hline(2, 3, 14) & 13 & 13 & -1 & 1 & \href{https://www.lmfdb.org/NumberField/3.3.49.1}{\Q(\lambda_7)}&3
    \\\hline(2, 3, 14) & 29 & 29 & 1 & 1 &\href{https://www.lmfdb.org/NumberField/3.3.49.1}{\Q(\lambda_7)}&3
    \\\hline(2, 3, 14) & 43 & 43 & -1 & 1 & \href{https://www.lmfdb.org/NumberField/3.3.49.1}{\Q(\lambda_7)}&3
    \\\hline(2, 3, 15) & 31 & 31 & 1 & 2 &\href{https://www.lmfdb.org/NumberField/4.4.1125.1}{\Q(\lambda_{15})}&4
    \\\hline(2, 3, 16) & 17 & 17 & -1 & 1 & \href{https://www.lmfdb.org/NumberField/4.4.2048.1}{\Q(\lambda_{16})}&4
    \\\hline(2, 3, 17) & 2 & 16 & 1 & 2 & \href{https://www.lmfdb.org/NumberField/8.8.410338673.1}{\Q(\lambda_{17})}&2
    \\\hline(2, 3, 17) & 17 & 17 & 1 & 2 &\href{https://www.lmfdb.org/NumberField/8.8.410338673.1}{\Q(\lambda_{17})}&1
    \\\hline(2, 3, 18) & 37 & 37 & 1 & 1 & \href{https://www.lmfdb.org/NumberField/3.3.81.1}{\Q(\lambda_9)}&3
    \\\hline(2, 3, 19) & 19 & 19 & 1 & 3\,(1)& \href{https://www.lmfdb.org/NumberField/9.9.16983563041.1}{\Q(\lambda_{19})}&1
    \\\hline(2, 3, 20) & 19 & 19 & -1 & 2 & \href{https://www.lmfdb.org/NumberField/4.4.2000.1}{\Q(\lambda_{20})}&4
    \\\hline(2, 3, 22) & 23 & 23 & -1 & 2 & \href{https://www.lmfdb.org/NumberField/5.5.14641.1}{\Q(\lambda_5)}&5
    \\\hline(2, 3, 24) & 5 & 25 & -1 & 1 & \href{https://www.lmfdb.org/NumberField/4.4.2304.1}{\Q(\sqrt{2},\sqrt{3})}&2
    \\\hline(2, 3, 26) & 3 & 27 & -1 & 2 & \href{https://www.lmfdb.org/NumberField/6.6.371293.1}{\Q(\lambda_{13})}&2
    \\\hline(2, 3, 30) & 31 & 31 & -1 & 1 & \href{https://www.lmfdb.org/NumberField/4.4.1125.1}{\Q(\lambda_{15})}&4
    \\
    \multicolumn{6}{c}{\vdots}
    \end{array}
    \egroup
    \\ 
    \end{equation*}
    
    \vfill\newpage
    
    \begin{equation*}
    \bgroup
\def\arraystretch{1.15}
\begin{array}{c|c||c|c||c|c|c}
    \bm{(a,b,c)}& \bm{p} &\bm{q}&\bm{{\rm PXL}} & \bm{{\rm adim}} &\bm{ E} &\bm{\#\text{ of }\frakp}
    \\\hline\hline(2, 4, 5) & 19 & 19 & -1 & 1& \href{https://www.lmfdb.org/NumberField/2.2.5.1}{\Q(\sqrt{5})}&2
    \\\hline(2, 4, 5) & 29 & 29 & -1 & 1 & \href{https://www.lmfdb.org/NumberField/2.2.5.1}{\Q(\sqrt{5})}&2
    \\\hline(2, 4, 5) & 31 & 31 & 1 & 1 &\href{https://www.lmfdb.org/NumberField/2.2.5.1}{\Q(\sqrt{5})}&2
    \\\hline(2, 4, 5) & 7 & 49 & 1 & 1 &\href{https://www.lmfdb.org/NumberField/2.2.5.1}{\Q(\sqrt{5})}&1
    \\\hline(2, 4, 5) & 61 & 61 & -1 & 1 & \href{https://www.lmfdb.org/NumberField/2.2.5.1}{\Q(\sqrt{5})}&2
    \\\hline(2, 4, 6) & 11 & 11 & -1 & 1 & \href{https://www.lmfdb.org/NumberField/1.1.1.1}{\Q}&1
    \\\hline(2, 4, 6) & 17 & 17 & -1 & 1 & \href{https://www.lmfdb.org/NumberField/1.1.1.1}{\Q}&1
    \\\hline(2, 4, 6) & 19 & 19 & -1 & 1 & \href{https://www.lmfdb.org/NumberField/1.1.1.1}{\Q}&1
    \\\hline(2, 4, 6) & 29 & 29 & -1 & 1 & \href{https://www.lmfdb.org/NumberField/1.1.1.1}{\Q}&1
    \\\hline(2, 4, 6) & 31 & 31 & -1 & 1 & \href{https://www.lmfdb.org/NumberField/1.1.1.1}{\Q}&1
    \\\hline (2, 4, 6) & 37 & 37 & -1 & 1 & \href{https://www.lmfdb.org/NumberField/1.1.1.1}{\Q}&1
    \\\hline (2, 4, 7) & 7 & 7 & 1 & 1 & \href{https://www.lmfdb.org/NumberField/3.3.49.1}{\Q(\lambda_7)}&1
    \\\hline(2, 4, 7) & 13 & 13 & -1 & 1 & \href{https://www.lmfdb.org/NumberField/3.3.49.1}{\Q(\lambda_7)}&3
    \\\hline(2, 4, 7) & 29 & 29 & -1 & 1 &\href{https://www.lmfdb.org/NumberField/3.3.49.1}{\Q(\lambda_7)}&3
    \\\hline(2, 4, 8) & 7 & 7 & -1 & 1 & \href{https://www.lmfdb.org/NumberField/2.2.8.1}{\Q(\sqrt{8})}&2
    \\\hline(2, 4, 8) & 5 & 25 & -1 & 1 &\href{https://www.lmfdb.org/NumberField/2.2.8.1}{\Q(\sqrt{8})}&1
    \\\hline(2, 4, 9) & 17 & 17 & 1 & 2 & \href{https://www.lmfdb.org/NumberField/3.3.81.1}{\Q(\lambda_9)}&3
    \\\hline(2, 4, 9) & 19 & 19 & -1 & 2 & \href{https://www.lmfdb.org/NumberField/3.3.81.1}{\Q(\lambda_9)}&3
    \\\hline(2, 4, 10) & 3 & 9 & -1 & 1 & \href{https://www.lmfdb.org/NumberField/2.2.5.1}{\Q(\sqrt{5})}&1
    \\\hline(2, 4, 10) & 11 & 11 & -1 & 1 & \href{https://www.lmfdb.org/NumberField/2.2.5.1}{\Q(\sqrt{5})}&2
    \\\hline(2, 4, 11) & 11 & 11 & -1 & 2\,(1) & \href{https://www.lmfdb.org/NumberField/5.5.14641.1}{\Q(\lambda_5)}&1
    \\\hline(2, 4, 12) & 5 & 25 & 1 & 1 & \href{https://www.lmfdb.org/NumberField/2.2.12.1}{\Q(\sqrt{12})}&1
    \\\hline(2, 4, 13) & 13 & 13 & -1 & 3\,(1) & \href{https://www.lmfdb.org/NumberField/6.6.371293.1}{\Q(\lambda_{13})}&1
    \\\hline(2, 4, 14) & 13 & 13 & -1 & 2 & \href{https://www.lmfdb.org/NumberField/3.3.49.1}{\Q(\lambda_7)}&3
    \\\hline(2, 4, 16) & 17 & 17 & -1 & 2 & \href{https://www.lmfdb.org/NumberField/4.4.2048.1}{\Q(\lambda_{16})}&4
    \\\hline(2, 4, 17) & 17 & 17 & 1 & 4\,(1) & \href{https://www.lmfdb.org/NumberField/8.8.410338673.1}{\Q(\lambda_{17})}&1
    \\\hline(2,5,5) & 3 & 9 & 1 & 1 & \href{https://www.lmfdb.org/NumberField/2.2.5.1}{\Q(\sqrt{5})} &1
    \\\hline(2,5,5) & 31 & 31 & 1 & 1 & \href{https://www.lmfdb.org/NumberField/2.2.5.1}{\Q(\sqrt{5})}&2
    \\\hline(2,5,5) & 41 & 41 & 1 & 1 & \href{https://www.lmfdb.org/NumberField/2.2.5.1}{\Q(\sqrt{5})} &2
    \\\hline(2, 5, 6) & 5 & 5 & -1 &1 & \href{https://www.lmfdb.org/NumberField/2.2.5.1}{\Q(\sqrt{5})}&1
    \\\hline(2, 5, 6) & 11 & 11 & 1 & 1 &\href{https://www.lmfdb.org/NumberField/2.2.5.1}{\Q(\sqrt{5})}&2
    \\\hline(2, 5, 6) & 19 & 19 & -1 & 1&\href{https://www.lmfdb.org/NumberField/2.2.5.1}{\Q(\sqrt{5})}&2
    \\\hline(2, 5, 6) & 31 & 31 & -1 & 1 &\href{https://www.lmfdb.org/NumberField/2.2.5.1}{\Q(\sqrt{5})}&2
    \\\hline(2, 5, 8) & 3 & 9 & -1 & 1 & \href{https://www.lmfdb.org/NumberField/4.4.1600.1}{\Q(\sqrt{2},\sqrt{5})}&2
    \\\hline(2, 5, 11) & 11 & 11 & 1 & 4 \,(2)& \href{https://www.lmfdb.org/NumberField/10.10.669871503125.1}{\Q(\sqrt{5},\lambda_{11})}&2
    \\\hline(2, 5, 12) & 11 & 11 & -1 & 2 & \href{https://www.lmfdb.org/NumberField/4.4.3600.1}{\Q(\sqrt{3},\sqrt{5})}&4
    \\\hline(2, 5, 15) & 2 & 16 & 1 & 2 & \href{https://www.lmfdb.org/NumberField/4.4.1125.1}{\Q(\lambda_{15})}&1
    \\\hline(2, 6, 6) & 5 & 5 & -1 & 1 & \href{https://www.lmfdb.org/NumberField/1.1.1.1}{\Q}&1
    \\\hline(2, 6, 6) & 19 & 19 & -1 & 1 & \href{https://www.lmfdb.org/NumberField/1.1.1.1}{\Q}&1
    \\\hline(2, 6, 7) & 13 & 13 & 1 & 2 & \href{https://www.lmfdb.org/NumberField/3.3.49.1}{\Q(\lambda_7)}&3
    \\\hline(2, 6, 8) & 7 & 7 & -1 & 1 & \href{https://www.lmfdb.org/NumberField/2.2.8.1}{\Q(\sqrt{8})}&2
    \\\hline(2, 6, 9) & 19 & 19 & -1 & 2 & \href{https://www.lmfdb.org/NumberField/3.3.81.1}{\Q(\lambda_9)}&3
    \\
    \multicolumn{6}{c}{\vdots}
        \end{array}
        \egroup
\end{equation*}

\vfill\newpage
    
\begin{equation*}
\bgroup
\def\arraystretch{1.15}
\begin{array}{c|c||c|c||c|c|c}
    \bm{(a,b,c)}& \bm{p} &\bm{q}&\bm{{\rm PXL}} & \bm{{\rm adim}} &\bm{ E}&\bm{\#\text{ of }\frakp}
    \\\hline\hline(2, 6, 10) & 11 & 11 & -1 & 2 &\href{https://www.lmfdb.org/NumberField/2.2.5.1}{\Q(\sqrt{5})}&2
    \\\hline(2, 6, 12) & 13 & 13 & -1 & 1 &\href{https://www.lmfdb.org/NumberField/2.2.12.1}{\Q(\sqrt{12})}&2
    \\\hline(2, 6, 13) & 13 & 13 & 1 & 4\,(1) & \href{https://www.lmfdb.org/NumberField/6.6.371293.1}{\Q(\lambda_{13})}&1
    \\\hline(2, 7, 7) & 7 & 7 & 1 & 1 &\href{https://www.lmfdb.org/NumberField/3.3.49.1}{\Q(\lambda_7)}&1
    \\\hline(2, 7, 8) & 7 & 7 & -1 & 2 & \href{https://www.lmfdb.org/NumberField/6.6.1229312.1}{\Q(\sqrt{2},\lambda_7)}&2
    \\\hline(2, 7, 9) & 2 & 8 & 1 & 3 & \href{https://www.lmfdb.org/NumberField/9.9.62523502209.1}{\Q(\lambda_7,\lambda_9)}&3
    \\\hline(2, 8, 8) & 17 & 17 & 1 & 1 &\href{https://www.lmfdb.org/NumberField/2.2.8.1}{\Q(\sqrt{8})}&2
    \\\hline(2, 8, 10) & 3 & 9 & -1 & 3 &\href{https://www.lmfdb.org/NumberField/4.4.1600.1}{\Q(\sqrt{2},\sqrt{5})}&2
    \\\hline(2, 10, 10) & 11 & 11 & -1 & 1 & \href{https://www.lmfdb.org/NumberField/2.2.5.1}{\Q(\sqrt{5})}&2
    \\\hline(2, 10, 11) & 11 & 11 & -1 & 6 \,(2)& \href{https://www.lmfdb.org/NumberField/10.10.669871503125.1}{\Q(\sqrt{5},\lambda_{11})}&2
    \\\hline(2, 12, 12) & 13 & 13 & -1 & 1 & \href{https://www.lmfdb.org/NumberField/2.2.12.1}{\Q(\sqrt{12})}&2  \\\hline(3, 3, 4) & 17 & 17 & 1 & 1 &\href{https://www.lmfdb.org/NumberField/2.2.8.1}{\Q(\sqrt{8})}&2
    \\\hline(3, 3, 4) & 31 & 31 & 1 & 1 &\href{https://www.lmfdb.org/NumberField/2.2.8.1}{\Q(\sqrt{8})}&2
    \\\hline(3,3,5)& 3 & 9 & 1& 1 & \href{https://www.lmfdb.org/NumberField/2.2.5.1}{\Q(\sqrt{5})}&1
    \\\hline(3,3,5) & 11 & 11 & 1 & 1 & 
    \href{https://www.lmfdb.org/NumberField/2.2.5.1}{\Q(\sqrt{5})} &2
    \\\hline(3,3,5) & 19 & 19 & 1 & 1 & 
    \href{https://www.lmfdb.org/NumberField/2.2.5.1}{\Q(\sqrt{5})} &2
    \\\hline(3,3,5) & 31 & 31 & 1 & 1 & 
    \href{https://www.lmfdb.org/NumberField/2.2.5.1}{\Q(\sqrt{5})} &2
    \\\hline(3, 3, 6) & 5 & 25 & 1 &1 & \href{https://www.lmfdb.org/NumberField/2.2.12.1}{\Q(\sqrt{12})}&1
    \\\hline(3, 3, 7) & 2 & 8 & 1 & 1 & \href{https://www.lmfdb.org/NumberField/3.3.49.1}{\Q(\lambda_7)}&1
    \\\hline(3, 3, 7) & 13 & 13 & 1 & 1 & \href{https://www.lmfdb.org/NumberField/3.3.49.1}{\Q(\lambda_7)}&3
    \\\hline(3, 3, 9) & 19 & 19 & 1 & 1 & \href{https://www.lmfdb.org/NumberField/3.3.81.1}{\Q(\lambda_9)}&3
    \\\hline(3, 3, 15) & 2 & 16 & 1 & 1 & \href{https://www.lmfdb.org/NumberField/4.4.1125.1}{\Q(\lambda_{15})}&1
    \\\hline(3,4,4) & 7 & 7 & 1 & 1 & \href{https://www.lmfdb.org/NumberField/1.1.1.1}{\Q} &1
    \\\hline(3,4,4) & 17 & 17 & 1 & 1 & \href{https://www.lmfdb.org/NumberField/1.1.1.1}{\Q} &1
    \\\hline(3, 4, 5) & 3 & 9 & 1 & 2 & \href{https://www.lmfdb.org/NumberField/4.4.1600.1}{\Q(\sqrt{5},\sqrt{8})}&2
    \\\hline(3, 4, 6) & 5 & 5 & -1 & 1 & \href{https://www.lmfdb.org/NumberField/2.2.24.1}{\Q(\sqrt{24})}&2
    \\\hline(3, 4, 7) & 7 & 7 & 1 & 2 & \href{https://www.lmfdb.org/NumberField/6.6.1229312.1}{\Q(\sqrt{2},\lambda_7)}&2
    \\\hline(3,4,12) & 13 & 13 & -1 & 1 & \href{https://www.lmfdb.org/NumberField/2.2.11.1}{\Q(\sqrt{3})} &2
    \\\hline(3, 5, 5) & 2 & 4 & 1 & 1 & \href{https://www.lmfdb.org/NumberField/2.2.5.1}{\Q(\sqrt{5})}&1
    \\\hline(3,5,5) & 5 & 5 & 1 & 1 & \href{https://www.lmfdb.org/NumberField/2.2.5.1}{\Q(\sqrt{5})}&1 
    \\\hline(3,5,5) & 11 & 11 & 1 & 1 & \href{https://www.lmfdb.org/NumberField/2.2.5.1}{\Q(\sqrt{5})} &2
    \\\hline(3,6,6) & 13 & 13 & 1 & 1 & \href{https://www.lmfdb.org/NumberField/1.1.1.1}{\Q} &1
    \\\hline(3, 6, 8) & 7 & 7 & -1 & 3 & \href{https://www.lmfdb.org/NumberField/4.4.1®8432.1}{4.4.18432.1}&4
    \\\hline(3, 7, 7) & 7 & 7 & 1 & 2 \,(1)& \href{https://www.lmfdb.org/NumberField/3.3.49.1}{\Q(\lambda_7)}&1
    \\\hline(3, 7, 7) & 2 & 8 & 1 & 2\,(1) & \href{https://www.lmfdb.org/NumberField/3.3.49.1}{\Q(\lambda_7)}&1
    \\\hline(4, 4, 4) & 17 & 17 & 1 & 1 & \href{https://www.lmfdb.org/NumberField/2.2.8.1}{\Q(\sqrt{8})}&2
    \\\hline(4, 4, 5) & 3 & 9 & 1 & 1 & \href{https://www.lmfdb.org/NumberField/2.2.5.1}{\Q(\sqrt{5})}&1
    \\\hline(4, 4, 6) & 13 & 13 & -1 & 1 & \href{https://www.lmfdb.org/NumberField/2.2.12.1}{\Q(\sqrt{12})}&2
    \\\hline(4, 5, 6) & 5 & 5 & -1 & 2 & \href{https://www.lmfdb.org/NumberField/4.4.14400.1}{\Q(\sqrt{5},\sqrt{24})}&2
    \\\hline(4, 6, 6) & 7 & 7 & -1 & 1 & \href{https://www.lmfdb.org/NumberField/2.2.8.1}{\Q(\sqrt{8})}&2
    \\\hline(4,8,8) & 3 & 9 & -1 & 1 & \href{https://www.lmfdb.org/NumberField/2.2.8.1}{\Q(\sqrt{2})} &1
    \\\hline(5,5,5) & 11 & 11 & 1 & 1 & \href{https://www.lmfdb.org/NumberField/2.2.5.1}{\Q(\sqrt{5})} &2
    \\\hline(7, 7, 7) & 2 & 8 & 1 & 1 & \href{https://www.lmfdb.org/NumberField/3.3.49.1}{\Q(\lambda_7)}&1
    \end{array}
    \egroup
\end{equation*}
\end{footnotesize}

\end{document}